\documentclass[12pt,a4paper,pdftex]{amsart}
\usepackage[margin=2.6cm]{geometry}
\usepackage[T1]{fontenc}
\usepackage[utf8]{inputenc}
\usepackage{lmodern}
\usepackage{amssymb}
\usepackage{tikz-cd}
\numberwithin{equation}{section}

\newtheorem{thm}{Theorem}[section]
\newtheorem{prop}[thm]{Proposition}
\newtheorem{lem}[thm]{Lemma}
\newtheorem{cor}[thm]{Corollary}

\theoremstyle{definition}
\newtheorem{defn}[thm]{Definition}
\newtheorem{assump}[thm]{Assumption}
\theoremstyle{remark}
\newtheorem{rem}[thm]{Remark}
\newtheorem*{claim}{Claim}
\DeclareMathOperator{\grk}{grk}
\DeclareMathOperator{\Hom}{Hom}
\DeclareMathOperator{\End}{End}
\DeclareMathOperator{\Ima}{Im}
\DeclareMathOperator{\supp}{supp}
\DeclareMathOperator{\Kar}{Kar}
\newcommand*{\Id}{\mathrm{Id}}
\newcommand*{\rex}{\mathrm{rex}}
\DeclareMathOperator{\ch}{ch}
\newcommand*{\ModCat}[1]{#1\textrm{\normalfont -Mod}}

\newcommand{\Z}{\mathbb{Z}}
\newcommand{\Coef}{\mathbb{K}}
\newcommand{\BSbimod}{\mathcal{BS}}
\newcommand{\Sbimod}{\mathcal{S}\mathrm{Bimod}}
\newcommand{\LL}{\mathit{LL}}
\newcommand{\DLL}{\mathbb{LL}}
\DeclareMathOperator{\Gr}{Gr}
\DeclareMathOperator{\Frac}{Frac}

\title{On Soergel bimodules}
\author{Noriyuki Abe}
\address{Graduate School of Mathematical Sciences, the University of Tokyo, 3-8-1 Komaba, Meguro-ku, Tokyo 153-8914, Japan.}
\email{abenori@ms.u-tokyo.ac.jp}

\begin{document}
\begin{abstract}
For a Coxeter system and a representation $V$ of this Coxeter system, Soergel defined a category which is now called the category of Soergel bimodules and proved that this gives a categorification of the Hecke algebra when $V$ is reflection faithful.
Elias and Williamson defined another category even when $V$ is not reflection faithful and they proved that this category is equivalent to the category of Soergel bimodules when $V$ is reflection faithful.
Moreover they proved the categorification theorem for their category with less assumptions on $V$.
In this paper, we give a ``bimodule theoretic'' definition of the category of Elias-Williamson and reprove the categorification theorem.
\end{abstract}
\maketitle

\section{Introduction}
In \cite{MR1029692}, Soergel gave a combinatorial description of the category $\mathcal{O}$ for semisimple Lie algebras.
This celebrated work has many applications, another proof of Kazhdan-Lusztig conjecture (which do not rely on the theory of $\mathcal{D}$-modules), Koszul duality of the category $\mathcal{O}$, etc.
Later, Soergel \cite{MR2329762} defined a certain category purely in terms of combinatorics of Coxeter systems, without any representation theory.
This category describes the category $\mathcal{O}$.
More precisely, this category is equivalent to the category of projective modules in (to be precisely, deformed version of) the principal block of the category $\mathcal{O}$ (cf.\ \cite{MR1029692,MR2205072}).

Let $(W,S)$ be a Coxeter system and $V$ \emph{reflection faithful} representation of $W$.
Then Soergel attached the category of Soergel bimodules.
When $W$ is the Weyl group of a semisimple Lie algebra, we may take $V$ to be a Cartan subalgebra and Soergel's category is the category which describes the category $\mathcal{O}$.

Fiebig used this category (or, more precisely, the category of sheaves on moment graphs, which is equivalent to the category of Soergel bimodules \cite{MR2395170}) to give an alternative proof of Lusztig conjecture which says that the irreducible characters of an algebraic group over a positive characteristic is given by affine Kazhdan-Lusztig polynomials if the characteristic is large enough.
He used the Soergel bimodules attached to an affine Weyl group $(W,S)$ and a Cartan subalgebra $V$ of corresponding affine Lie algebra.
If the coefficient field of $V$ has the characteristic zero, then $V$ is reflection faithful.
However if the characteristic is positive, then this is not reflection faithful.
So he used a lifting to characteristic zero and used the theory of Soergel bimodules over characteristic zero field.

However, of course, it is more natural to use positive characteristic objects directly.
Elias and Williamson gave an alternative category of the category of Soergel bimodules which works well even with non-reflection faithful representation \cite{MR3555156}.
Riche and Williamson \cite{MR3805034} gave a conjecture which claims that this category describes the category of algebraic representations of an algebraic group over any field of positive characteristic.
As an application, this description gives a character formula of tilting modules in terms of $p$-canonical basis defined using the category of Elias-Williamson.
Recently this character formula was proved by Achar-Makisumi-Riche-Williamson \cite{MR3868004} for the principal block when $p$ is greater than the Coxeter number.

Elias and Williamson gave the definition of their category as a certain diagrammatic category with  generators and (very complicated) relations.
Such generators and relations appeared in a study of Soergel bimodules, however the definition seems completely different from the original definition of Soergel bimodules.
The aim of this paper is to give a ``bimodule theoretic'' definition of this category, namely, we define a new category and prove that the new category is equivalent to that of Elias-Williamson.

To say more precisely, we first recall the category of Soergel bimodules.
Let $R = S(V)$ be the symmetric algebra of $V$ and for $s\in S$, let $R^s$ be the subalgebra of $s$-invariants.
Let $(n)$ be a shift of grading defined by $M(n)^i = M^{i + n}$ for $n\in\Z$ where $M = \bigoplus_i M^i$ is a graded module.
Then a Soergel bimodule is a graded $R$-bimodule which is a direct summand of a direct sum of modules of a form
\begin{equation}\label{eq:BS module, in Introduction}
R\otimes_{R^{s_1}}R\otimes_{R^{s_2}}\cdots\otimes_{R^{s_l}}R(n)
\end{equation}
for $s_1,\dots,s_l\in S$ and $n\in\Z$.
Let $\Sbimod$ be the category of Soergel bimodules and $[\Sbimod]$ the spliit Grothendieck group of $\Sbimod$.
Then Soergel proved the following which we call Soergel's categorification theorem.
\begin{enumerate}
\item For each $w\in W$, there exists a unique indecomposable module $B(w)\in \Sbimod$ which satisfies the following.
\begin{enumerate}
\item For a reduced expression $w = s_1\dotsm s_l$, $B(w)$ appears as a direct summand of $R\otimes_{R^{s_1}}\cdots\otimes_{R^{s_l}}R(l)$ with multiplicity one and the module $R\otimes_{R^{s_1}}\cdots\otimes_{R^{s_l}}R(l)$ is a direct sum of $B(v)(l)$ where $v < w$ and $l\in\Z$.
\item Any object in $\Sbimod$ is a direct sum of modules $B(w)(n)$ ($w\in W$, $n\in\Z$).
\end{enumerate}
\item The algebra $[\Sbimod]$ is isomorphic to the Hecke algebra attached to $(W,S)$.
\end{enumerate}
We will extend this theorem.

\subsection{A representation $V$ and our category}
Let $V$ be a finite-dimensional representation of $W$ over a field $\Coef$.
In the Introduction, we assume the following.
There exist $\alpha_s\in V$ and $\alpha_s^\vee\in V^*$ (where $V^*$ is the dual of $V$) for each $s\in S$ such that 
\begin{enumerate}
\item $\langle \alpha^\vee_s,\alpha_s\rangle = 2$ for any $s\in S$.
\item $s(v) = v - \langle \alpha_s^\vee,v\rangle \alpha_s$ for any $s\in S$ and $v\in V$.
\item $\alpha_s\ne 0$ and $\alpha_s^\vee\ne 0$.
\item For each $s,t\in S$ such that the order of $st$ is finite, the representation of $V$ restricted to the group generated by $\{s,t\}$ is reflection faithful.
\end{enumerate}
Under these assumptions, we consider the following category.
Let $Q$ be the fraction field of $R$.
Our category consists of $M$ which satisfy
\begin{quote}
$M$ is a graded $R$-bimodule with a decomposition $M\otimes_R Q = \bigoplus_{w\in W}M_Q^w$ such that $f\in R$ and $m\in M_Q^w$ we have $mf = w(f)m$.
\end{quote}
The morphisms we consider are degree zero homomorphisms as $R$-bimodules which upon tensoring with $Q$ preserve the given decompositions.
It is easy to see that the $R$-bimodule \eqref{eq:BS module, in Introduction} naturally has such a decomposition and we say $M$ is a Soergel bimodule if it appears as a direct summand of a sum of modules of this type.
The main theorem of this paper is the following.
\begin{thm}
\begin{enumerate}
\item For this category we have Soergel's categorification theorem.
\item This category is equivalent to the category of Elias-Williamson.
\end{enumerate}
\end{thm}

\subsection{Sheaves on moment graphs}
As we mentioned in the above, Fiebig used sheaves on moment graphs to give an alternative proof of Lusztig conjecture \cite{MR2726602} for sufficiently large primes.
In this paper, we prove the following.
\begin{thm}
Assume that the GKM condition holds.
Our category is equivalent to a certain full subcategory of $\mathcal{Z}$-modules (see \ref{subsec:Sheaves on moment graphs} for the precise definition) where $\mathcal{Z}$ is the structure algebra of the moment graph attached to $(W,S)$.
\end{thm}
For the GKM condition, see \ref{subsec:Sheaves on moment graphs}.
As a corollary, the category of Elias-Williamson is equivalent to the full subcategory of $\mathcal{Z}$-modules.
If $(W,S,V)$ is coming from a Kac-Moody group, this is proved using parity sheaves by combining works of Fiebig-Williamson \cite{MR3330913} and Riche-Williamson \cite{MR3805034}.

\subsection{Proof}
Even though the definitions and theorems are similar to that of Soergel \cite{MR2329762}, it seems difficult to follow his argument in our setting.
For example, he proved that after a suitable localization, the modules decomposes into modules attached to rank one \cite[Lemma~6.10]{MR2329762}.
However this does not hold in our case.
Fiebig used similar arguments in his study of moment graphs.
To use this argument, he assumed GKM condition on the moment graph.
This condition does not follow from our assumptions on $V$.

There is another point in Soergel's argument which we cannot apply to our case.
He considered the ``standard module'' $\Delta_x$ for each $x\in W$ and proved that there exists an extension between $\Delta_x$ and $\Delta_y$ only when $x^{-1}y$ is a reflection (a conjugation of an element in $S$).
Therefore the category $\mathcal{F}_\Delta$ consisting of the objects which admit a ``standard flag'' behaves well.
However in our case, there are more extensions and the analogue of the category $\mathcal{F}_\Delta$ does not seem to behave well.

We analyze the modules in \eqref{eq:BS module, in Introduction} directly using light leaves introduced by Libedinsky \cite{MR2441994}.
Using Soergel's theorem, Libedinsky proved that the light leaves give a basis of a certain space of homomorphism.
In this paper, we prove Libedinsky's result directly and use it to prove Soergel's categorification theorem.
The argument is new even for the original case.

\subsection{Application}
Since our definition of the category is a direct generalization of that of Soergel, we can hope that many works with Soergel bimodules can be generalized to our setting with a few modification of arguments.
This should be the biggest application of our theory.

We also have the following concrete application.
Let $G$ be a connected reductive group defined over a positive characteristic field and $T$ a maximal torus.
Then we give an action of our category (or equivalently, the category of Elias-Williamson) attached to the affine Weyl group of $G$ on a regular block of $G_1T$-modules where $G_1$ is the Frobenius kernel~\cite{arXiv:1904.11350}.
This is a $G_1T$-version of a conjecture of Riche-Williamson~\cite{MR3805034}.

\subsection{Organization of the paper}
In the next section, we introduce our category and give basic properties of it.
We also introduce the notation on Hecke algebras.
In Section~\ref{sec:Light leaves}, we recall the definition of light leaves.
Using the light leaves, we prove freeness of a certain module and calculate its graded rank.
In Section~\ref{sec:The categorification theorem}, we prove the categorification theorem based on the theorems in Section~\ref{sec:Light leaves}.
In the final section, we compare our category with the other categories.

\subsection*{Acknowledgment}
The author was supported by JSPS KAKENHI Grant Number 18H01107.
The author thanks Henning Haahr-Andersen and Masaharu Kaneda for reading the paper and giving helpful comments.

\section{The category}\label{sec:The category}
We follow notation in \cite{MR3555156}.

\subsection{A representation}
Throughout this paper, let $(W,S)$ be a Coxeter system such that $\#S < \infty$ and $\Coef$ a noetherian integral domain.
The length function $W\to  \Z_{\ge 0}$ of $W$ is denoted by $\ell$ and the Bruhat order on $W$ is denoted by $\le$.

Let $(V,\{\alpha_s\}_{s\in S},\{\alpha_s^\vee\}_{s\in S})$ be a triple such that
\begin{itemize}
\item $V$ is a free $\Coef$-module of finite rank with a $\Coef$-linear action of $W$.
\item $\alpha_s\in V$, $\alpha_s^\vee\in V^*$ where $V^* = \Hom_{\Coef}(V,\Coef)$.
\end{itemize}
We assume that this satisfies:
\begin{enumerate}
\item $\langle\alpha_s^\vee,\alpha_s\rangle = 2$ for each $s\in S$.
\item $s(v) = v - \langle\alpha_s^\vee,v\rangle\alpha_s$ for $s\in S$ and $v\in V$.
\item $\alpha_s^\vee\colon V\to \Coef$ is surjective and $\alpha_s\ne 0$ for any $s\in S$.
\end{enumerate}
Note that the third condition follows from the first condition if $2\in \Coef^\times$.
Later we will add one more assumption (Assumption~\ref{assump:existecne of 2-coloerd map}).
Set $R = S(V)$ and let $Q$ be the fraction field of $R$.
We regard $R$ as a graded algebra via $\deg(V) = 2$.
The group $W$ on $V$ extends uniquely on $R$ by automorphisms of algebras.
We remark that $R$ is a noetherian integral domain.

We call $t\in W$ a reflection if it is conjugate to an element in $S$.
Let $t = wsw^{-1}$ be a reflection where $s\in S$ and $w\in W$.
Set $\alpha_t = w(\alpha_s)$.
This depends on a choice of $s,w$ in general, so we fix such $s,w$ to define $\alpha_t$.
By the following lemma, $\Coef^\times \alpha_t$ does not depend on the choice of $s,w$.

\begin{lem}
If $wsw^{-1} = s'$ where $s,s'\in S$ and $w\in W$, then $\alpha_{s'}\in \Coef^\times w(\alpha_s)$.
\end{lem}
\begin{proof}
Take $\delta\in V$ such that $\langle \alpha_s^\vee,\delta\rangle = 1$.
Then we have $s(\delta) = \delta - \alpha_s$.
Hence $s'(w(\delta)) = ws(\delta) = w(\delta) - w(\alpha_s)$.
On the other hand, we have $s'(w(\delta)) = w(\delta) - r\alpha_{s'}$ where $r = \langle \alpha_{s'}^\vee,w(\delta)\rangle\in\Coef$.
Hence $w(\alpha_s) = r\alpha_{s'}$.
Replacing $s,w,s'$ with $s',w^{-1},s$ respectively, there exists $r'$ such that $w^{-1}(\alpha_{s'}) = r'\alpha_s$.
Therefore $w(\alpha_s) = r\alpha_{s'} = rr'w(\alpha_s)$.
Since $V$ is free, $\Coef$ is an integral domain and $w(\alpha_s)\ne 0$, we have $rr' = 1$.
Therefore $r,r'\in \Coef^\times$.
\end{proof}

\subsection{The category}
First let $\mathcal{C}'$ be the category of graded $R$-bimodules $M$ with a decomposition $M\otimes_R Q = \bigoplus_{w\in W}M_Q^w$ as $(Q,R)$-bimodules such that 
\begin{itemize}
\item $M_Q^w \ne 0$ only for finite $w$.
\item For $f\in R$ and $m\in M_Q^w$, we have $mf = w(f)m$.
\end{itemize}
A homomorphism $\varphi\colon M\to N$ in $\mathcal{C}'$ is a degree zero bimodule homomorphism $M\to N$ which sends $M_Q^w$ to $N_Q^w$.
Note that $\Hom_{\mathcal{C}}(M,N)$ has a structure of $R$-bimodules via $(a\varphi b)(m) = \varphi(amb)$ for $\varphi\in \Hom_{\mathcal{C}}(M,N)$, $m\in M$ and $a,b\in R$.

\begin{rem}\label{rem:right Q-action}
By the second condition, the left action of $f\in R\setminus\{0\}$ on $M_Q^w$ is invertible.
Therefore it is also invertible on $M\otimes_R Q$ and therefore $M\otimes_R Q$ is a $Q$-bimodule.
The decomposition $M\otimes_R Q = \bigoplus_{w\in W}M_Q^w$ is the decomposition as $Q$-bimodules.
\end{rem}

\begin{rem}
If $V$ is a faithful $W$-representation, the decomposition $M\otimes_R Q = \bigoplus_{w\in W}M_Q^w$ is uniquely determined by the $Q$-bimodule structure (if exists) and each $R$-bimodule homomorphism preserves the decompositions.
Therefore in this case $\mathcal{C}'$ is a full subcategory of the category of $R$-bimodules.
\end{rem}

This is basically the category which we want to consider.
It is useful to add some more assumptions.
For $M\in \mathcal{C}'$, we say $M\in \mathcal{C}$ if $M$ is finitely generated as a $R$-bimodule and flat as a right $R$-module.
Since $M$ is flat as a right $R$-module, it is torsion free.
Hence we have $M\hookrightarrow M\otimes_R Q = \bigoplus_{w\in W}M_Q^w$.
In particular $M$ is torsion-free as a left $R$-module.
As in Remark~\ref{rem:right Q-action}, $f\notin R\setminus\{0\}$ acts invertible from the left on $M\otimes_R Q$.
Hence we have a homomorphism $Q\otimes_R M\hookrightarrow M\otimes_R Q$.
It is not difficult to see that this is an isomorphism.

A typical example of an object in $\mathcal{C}$ is the module which we will denote $R_w$.
For $w\in W$, let $R_w$ be an object of $\mathcal{C}$ defined as follows: as a left $R$-module, $R_w = R$ and the bimodule structure is given by $mf = w(f)m$ for $m\in R_w$ and $f\in R$.
The module $(R_w)_Q^x$ is given by
\[
(R_w)_Q^x = 
\begin{cases}
Q & (x = w),\\
0 & (x\ne w).
\end{cases}
\]
We denote $Q_w = R_w\otimes_R Q$ which is a $Q$-bimodule.

Let $M\in \mathcal{C}$.
\begin{itemize}
\item For $m\in M$, let $m_w$ be the image of $m$ under $M\to \bigoplus_{w\in W}M_Q^w\twoheadrightarrow M_Q^w$.
\item We set $M_Q = M\otimes_R Q$.
\item For $I\subset W$, let $M_I$ (resp.\ $M^I$) be the inverse image of $\bigoplus_{w\in I}M_Q^w$ in $M$ (resp.\ the image of $M$ in $\bigoplus_{w\in I}M_Q^w$). We can regard $M_I\subset M^I\subset \bigoplus_{w\in I}M_Q^w$.
\item For $w\in W$, we write $M_w$ (resp.\ $M^w$) for $M_{\{w\}}$ (resp.\ $M^{\{w\}}$).
\item We write $\supp_W(M) = \{w\in W\mid M_Q^w\ne 0\}$ and $\supp_W(m) = \{w\in W\mid m_w \ne 0\}$ for $m\in M$.
\end{itemize}
We have $M_I = \{m\in M\mid \supp_W(m)\subset I\}$.

\begin{lem}
The modules $M_I$ and $M^I$ are both objects in $\mathcal{C}'$ such that $(M_I)_Q = (M^I)_Q = \bigoplus_{w\in I}M_Q^w$.
\end{lem}
\begin{proof}
Since $M_I\hookrightarrow M^I\subset \bigoplus_{w\in I}M_Q^w$, we have $M_I\otimes Q\hookrightarrow M^I\otimes Q\subset \bigoplus_{w\in I}M_Q^w$.
Hence it is sufficient to prove that $M_I\otimes Q\hookrightarrow \bigoplus_{w\in I}M_Q^w$ is surjective.
Let $m\in \bigoplus_{w\in I}M_Q^w \subset M_Q$ and take $0\ne f\in R$ such that $fm \in M$.
Then we have $fm_w = (fm)_w = 0$ for any $w\in W\setminus I$.
Since $M_Q^w$ is torsion-free, $m_w = 0$.
Therefore $m\in (M_I)_Q$.
\end{proof}

The following lemma is clear from the definitions.

\begin{lem}\label{lem:left adjointness of upper I}
Let $I\subset W$ and $M,N\in \mathcal{C}$ such that $\supp_W(N)\subset I$.
Then we have $\Hom_{\mathcal{C}'}(M,N)\simeq \Hom_{\mathcal{C}'}(M^I,N)$ and $\Hom_{\mathcal{C}'}(N,M)\simeq \Hom_{\mathcal{C}'}(N,M_I)$.
\end{lem}

\begin{lem}\label{lem:finitey generated as left/right module}
Any $M\in \mathcal{C}$ is finitely generated as a left (resp.\ right) $R$-module.
\end{lem}
\begin{proof}
We only prove that $M$ is finitely generated as a left $R$-module.
Since $M^w$ is a quotient of $M$, this is also a finitely generated $R$-bimodule.
The formula $mf = w(f)m$ for $f\in R$, $m\in M^w$ says that $M^w$ is also finitely generated as a left $R$-module.
Since $M\hookrightarrow \bigoplus_{w\in W}M^w$ and $M^w \ne 0$ only for finite $w$, $M$ is a finitely generated left $R$-module.
\end{proof}

We define a tensor product $M\otimes N$ of $M,N\in \mathcal{C}$ as follows.
As an $R$-bimodule, $M\otimes N = M\otimes_R N$ and we attache the decomposition $(M\otimes N)_Q^w = \bigoplus_{xy = w}M_Q^x\otimes_Q N_Q^y$.
This gives a structure of a monoidal category to $\mathcal{C}$.
The unit object is $R_e$ where $e$ is the unit element of $W$.
The following is obvious from the definition.
\begin{lem}\label{lem:support of the tensor product}
We have $\supp_W(M\otimes N) = \{xy \mid x\in \supp_W(M),y\in \supp_W(N)\}$.
\end{lem}

\subsection{Notes on gradings}
Let $M = \bigoplus_{i\in \Z}M^i$ be a graded left $R$-module or right $R$-module or $R$-bimodule.
The grading shift $M(1)$ is defined as $M(1)^i = M^{i + 1}$.
Therefore, for degree $k$ element $f\in R$, the multiplication $m\mapsto fm$ gives a degree $0$ homomorphism $M\to M(k)$ and the submodule $fR$ is isomorphic to $R(-k)$.
If $M$ is a graded free left $R$-module, namely $M\simeq \bigoplus_i R(n_i)$ for $n_i\in \Z$, we define its graded rank $\grk(M)$ by $\grk(M) = \sum_i v^{n_i}$ where $v$ is an indeterminate.
Obviously we have $\grk(M_1\oplus M_2) = \grk(M_1) + \grk(M_2)$ and $\grk(M(1)) = v\grk(M)$.
If $M$ has a basis $\{m_i\}$, then $\grk(M) = \sum_i v^{-\deg(m_i)}$.

\begin{lem}
Let $M\in \mathcal{C}$, $w\in W$ and assume that $M^w$ is a graded free left $R$-module.
Then $M^w\simeq \bigoplus_i R_w(n_i)$ for some $n_i\in \Z$ as an object in $\mathcal{C}$.
\end{lem}
\begin{proof}
The assumption says that we have an isomorphism $M^w\simeq \bigoplus_i R_w(n_i)$ as left $R$-modules for some $n_i\in \Z$.
Since the right action of $f\in R$ is equal to the left action of $w(f)$ on both sides, this is an isomorphism as $R$-bimodules.
\end{proof}

\begin{lem}\label{lem:filtration, graded rank}
Assume that $\Coef$ is a field.
Let $M$ be a finitely generated graded free left $R$-module with the graded rank $p$ and $0 = M_0\subset M_1\subset M_2\subset \dots \subset M_r = M$ be a filtration as graded left $R$-modules.
Assume that there exists a graded free left $R$-submodule $N_i\subset M_i/M_{i - 1}$ with the graded rank $q^{(i)}$ such that $\sum_i q^{(i)} = p$.
Then we have $N_i = M_i/M_{i - 1}$.
\end{lem}
\begin{proof}
Let $N_i^l$ (resp.\ $M_i^l$, $M^l$) be the $l$-th graded piece of $N_i$.
Then we have $\sum_i\dim N_i^l = \dim M^l$ by the assumption and we have $\dim M^l = \sum \dim(M_i/M_{i - 1})^l$.
Hence $N_i^l = (M_i/M_{i - 1})^l$ for any $l$.
Therefore we have $N_i = M_i/M_{i - 1}$.
\end{proof}

\subsection{Soergel bimodules}
Let $s\in S$ and set $R^s = \{f\in R\mid s(f) = f\}$.
We define $B_s\in \mathcal{C}$ by $B_s = R\otimes_{R^s}R(1)$.
The $Q$-bimodule $(B_s)_Q^w$ is uniquely determined by the condition $\supp_W(B_s) = \{e,s\}$.
Fix $\delta_s\in V$ such that $\langle \alpha_s^\vee,\delta_s\rangle = 1$.
Then an explicit description of the $Q$-bimodule $(B_s)_Q^w$ is
\begin{equation}\label{eq:decomposition of B_s in Q}
\begin{split}
(B_s)_Q^e & = Q(\delta_s \otimes 1 - 1 \otimes s(\delta_s))\simeq Q_e,\\
(B_s)_Q^s & = Q(\delta_s \otimes 1 - 1 \otimes \delta_s)\simeq Q_s.
\end{split}
\end{equation}
The projection $B_s\to (B_s)_Q^e$ (resp.\ $B_s\to (B_s)_Q^s$) is given by $f\otimes g\mapsto fg$ (resp.\ $f\otimes g\mapsto fs(g)$).
To prove this description is straightforward using the following lemma.
\begin{lem}[{\cite[Claim~3.11]{MR3555156}}]\label{lem:basis over R^s}
We have $R = R^s\oplus \delta_s R^s$.
Therefore $B_s \simeq R(1)\oplus R(-1)$ as left or right $R$-modules.
\end{lem}

Let $M\in \mathcal{C}$.
Then we have $B_s\otimes M\in \mathcal{C}$.
We have $(B_s\otimes M)_Q^w = \bigoplus_{xy = w}(B_s)_Q^x\otimes M_Q^y = Q_e\otimes_Q M_Q^w\oplus Q_s\otimes_Q M_Q^{sw}$.
We have $Q_e\otimes M_Q^w\simeq M_Q^w$ and $Q_s\otimes M_Q^{sw}\simeq M_Q^{sw}$ as right $Q$-modules.
As left $Q$-modules, we have $Q_e\otimes M_Q^w\simeq M_Q^w$ but $Q_s\otimes M_Q^{sw}$ is not isomorphic to $M_Q^{sw}$.
The action on the left hand side is twisted by $s$, namely the action of $f\in R$ on $Q_s\otimes M_Q^{sw}$ is equal to the action of $s(f)$ on $M_Q^{sw}$.

As $R$-bimodules, $B_s\otimes M = R\otimes_{R^s}M$.
In terms of this description, $(B_s\otimes M)_Q^w$ is given as follows.
\begin{lem}\label{lem:stalk of B_s otimes M}
Let $s\in S$ and $M\in \mathcal{C}$.
\begin{enumerate}
\item The projection $R\otimes_{R^s}M = B_s\otimes M\to (B_s\otimes M)_Q^w = M_Q^w\oplus M_Q^{sw}$ is given by $f\otimes m_w\mapsto (fm,s(f)m_{sw})$.
\item We have $(B_s\otimes M)_Q^w = \{(\delta_s\otimes m - 1\otimes s(\delta_s)m) + (\delta_s\otimes m' - 1\otimes\delta_s m')\mid m\in M_Q^w,m'\in M_Q^{sw}\}$.
\end{enumerate}
\end{lem}
The following is easy, for example from the above lemma.
\begin{lem}\label{lem:support of B_s otimes}
Let $s\in S$ and $M\in \mathcal{C}$.
We have $(B_s\otimes M)_{Q}^w + (B_s\otimes M)_{Q}^{sw} = B_s\otimes_R M_Q^w + B_s\otimes_R M_Q^{sw}$.
\end{lem}

Let $\BSbimod$ be the full-subcategory of $\mathcal{C}$ such that the objects are $\{B_{s_1}\otimes\cdots\otimes B_{s_l}(n)\mid s_1,\dots,s_l\in S,n\in\Z\}$ and the category $\Sbimod$ is defined as a full-subcategory of $\mathcal{C}$ whose objects are direct summands of objects in the category $\BSbimod$.
Obviously these categories are stable under the tensor products.

For $\underline{x} = (s_1,\dots,s_l)\in S^l$, we put $B_{\underline{x}} = B_{s_1}\otimes\dotsm\otimes B_{s_l}$.
By Lemma~\ref{lem:support of the tensor product}, we get the following.
\begin{lem}
We have $\supp_W(B_{\underline{x}}) = \{s_1^{e_1}\dotsm s_l^{e_l}\mid e_i \in \{0,1\}\}$.
In particular, if $\underline{x}$ is a reduced expression of $x\in W$, then $\supp_W(B_{\underline{x}}) = \{y\in W\mid y\le x\}$.
\end{lem}

Since $B_s\otimes M \simeq M(-1) \oplus M(1)$ (resp.\ $M\otimes B_s\simeq M(-1)\oplus M(1)$) as a right (resp.\ left) $R$-module, we get the following.
\begin{lem}\label{lem:graded rank of B}
The module $B_{\underline{x}}$ is graded free as a left (resp.\ right) $R$-module and its graded rank is $(v + v^{-1})^l$.
\end{lem}

\begin{lem}\label{lem:B_s is self-adjoint}
Let $M,N\in \mathcal{C}$ and $s\in S$.
Then $\Hom_{\mathcal{C}}(B_s\otimes M,N)\simeq \Hom_{\mathcal{C}}(M,B_s\otimes N)$.
\end{lem}
\begin{proof}
We regard $B_s\otimes M = R\otimes_{R^s}M$ and $B_s\otimes N = R\otimes_{R^s}N$.
Take $\delta\in V$ such that $\langle\alpha_s^\vee,\delta\rangle = 1$.
Let $\varphi\colon B_s\otimes M\to N$ and define $\psi\colon M\to B_s\otimes N$ by $\psi(m) = 1\otimes\varphi(1\otimes \delta m) - s(\delta)\otimes\varphi(1\otimes m)$.
Then as in \cite[Lemma~3.3]{MR2441994}, $\psi$ is an $R$-bimodule homomorphism and this correspondence gives an isomorphism between the space of $R$-bimodule homomorphisms.
We prove that this correspondence preserves homomorphisms in $\mathcal{C}$.

Let $a(m) = 1\otimes \delta m - s(\delta)\otimes m $ and $b(m) = 1\otimes \delta m - \delta\otimes m$.
We have $a(m) = (\delta - (\delta + s(\delta)))\otimes m + 1\otimes \delta m = \delta\otimes m - 1\otimes (\delta + s(\delta))m + 1\otimes \delta m = \delta\otimes m - 1\otimes s(\delta)m$.
Hence $(B_s\otimes M)_Q^w = a(M_Q^w) + b(M_Q^{sw})$ by Lemma~\ref{lem:stalk of B_s otimes M}.
By the definition of $\psi$ and Lemma~\ref{lem:stalk of B_s otimes M}, the image of $\psi(m)$ in $(B_s\otimes N)_Q^w$ is $(\varphi(a(m))_w,\varphi(b(m))_{sw})$.
Therefore $\psi(m)$ is a homomorphism in $\mathcal{C}$ if and only if $\varphi(a(M_Q^w))_{y} = 0$ and $\varphi(b(M_Q^w))_{sy} = 0$ for any $w,y\in W$ such that $y\ne w$, if and only if $\varphi(a(M_Q^w))\subset N_Q^w$ and $\varphi(b(M_Q^{sw}))\subset N_Q^w$ for any $w\in W$, if and only if $\varphi$ is a homomorphism in $\mathcal{C}$.
\end{proof}

\subsection{The Hecke algebra}
In this paper, we use the following definition of the Hecke algebra.
Let $v$ be an indeterminate.
The $\Z[v^{\pm 1}]$-algebra $\mathcal{H}$ is generated by $\{H_w\mid w\in W\}$ and defined by the following relations.
\begin{itemize}
\item $(H_s - v^{-1})(H_s + v) = 0$ for any $s\in S$.
\item If $\ell(w_1) + \ell(w_2) = \ell(w_1w_2)$ for $w_1,w_2\in W$, we have $H_{w_1w_2} = H_{w_1}H_{w_2}$.
\end{itemize}
For $s\in S$, put $\underline{H}_s = H_s + v$ and for $\underline{x} = (s_1,\dots,s_l)\in S^l$, we put $\underline{H}_{\underline{x}} = \underline{H}_{s_1}\dotsm \underline{H}_{s_l}$.

It is known that $\{H_w\mid w\in W\}$ is a basis of a $\Z[v^{\pm 1}]$-module $\mathcal{H}$.
We define $p_{\underline{x}}^w\in \Z[v^{\pm 1}]$ by $\underline{H}_{\underline{x}} = \sum_{w\in W}p_{\underline{x}}^w H_w$.
\begin{lem}\label{lem:sum of p, for filtration}
We have $\sum_{w\in W}v^{\ell(w)}p_{\underline{x}}^w(v^{-1}) = (v + v^{-1})^{l}$.
\end{lem}
\begin{proof}
By the defining relations, $H_w\mapsto v^{-\ell(w)}$ gives an algebra homomorphism $\mathcal{H}\to \Z[v^{\pm 1}]$.
Applying this homomorphism to $\underline{H}_{\underline{x}} = \sum_{w\in W}p_{\underline{x}}^w H_w$, we get $(v + v^{-1})^l = \sum_{w\in W}p_{\underline{x}}^w(v)v^{-\ell(w)}$.
Replacing $v$ with $v^{-1}$, we get the lemma.
\end{proof}

\begin{lem}\label{lem:Q-dimension of BS module}
The dimension of $B_{\underline{x}}^w \otimes Q$ as a $Q$-vector space is $p_{\underline{x}}^w(1)$.
\end{lem}
\begin{proof}
After specializing $v$ to $1$, $\mathcal{H}$ is isomorphic to the group algebra $\Z[W]$ via $\mathcal{H}\ni H_w\mapsto w\in \Z[W]$.
Therefore we have $(s_1 + 1)\dotsm (s_l + 1) = \sum_w p_{(s_1,\dots,s_l)}^w(1)w$.
Hence we have:
\[
p_{(s,s_1,\dots,s_l)}^w(1) = p_{(s_1,\dots,s_l)}^w(1) + p_{(s_1,\dots,s_l)}^{sw}(1).
\]
On the other hand, we have $(B_s\otimes M)_Q^w = M_Q^w\oplus M_Q^{sw}$ for any $M\in \mathcal{C}$.
Hence we have $\dim_Q(B_s\otimes M)^w_Q = \dim_Q(M_Q^w) + \dim_Q  (M_Q^{sw})$.
Now we get the lemma by induction on the length of $\underline{x}$.
\end{proof}

\subsection{Duality}
Let $M\in \mathcal{C}$.
Define a new module $D(M)\in \mathcal{C}'$ by
\begin{align*}
D(M) & = \Hom_{\text{-$R$}}(M,R),\\
D(M)_Q^w & = \Hom_{\text{-$Q$}}(M_Q^w,Q).
\end{align*}
Here $\Hom_{\text{-$R$}}$ means the space of homomorphisms as right $R$-modules and $R$-bimodule structure is given by $(afb)(m) = f(amb)$ for $f\in D(M)$, $a,b\in R$ and $m\in M$.
Since $M$ is a finitely generated right $R$-module and we have an isomorphism $M\otimes_R Q\simeq M\otimes_R Q$, we have $D(M)\otimes Q  = \Hom_{\text{-$Q$}}(M_Q,Q) = \bigoplus_{w\in W}D(M)_Q^w$.
Therefore $D(M)\in \mathcal{C}'$.
For $\varphi\in D(M)$ and $w\in W$, $\varphi_w$ is the restriction of $\Id\otimes \varphi\colon M\otimes Q = \bigoplus_{w\in W}M_Q^w\to Q$ to $M_Q^w$.

\begin{lem}\label{lem:duality on local}
We have $D(M^I)\simeq D(M)_I$ for any $I\subset W$.
\end{lem}
\begin{proof}
Since $M^I$ is a quotient of $M$, we have $D(M^I)\subset D(M)$ and $\psi\in D(M)$ is in $D(M^I)$ if and only if $\psi$ is zero on $M_{W\setminus I}$.
Since $R$ is an integral domain, $\psi$ is zero on $M_{W\setminus I}$ if and only if $\Id\otimes \psi$ is zero on $M_{W\setminus I}\otimes Q  = \bigoplus_{w\in W\setminus I}M_Q^w$.
Namely $\psi\in D(M)_I$.
\end{proof}

\begin{lem}\label{lem:graded rank of the dual}
Let $M\in \mathcal{C}$ and $w\in W$ and assume that $M^w$ is graded free as a left $R$-module.
\begin{enumerate}
\item The module $D(M)_w$ is also a graded free left $R$-module and its graded rank is given by $\grk(D(M)_w)(v) = \grk(M^w)(v^{-1})$.
\item We have $D(D(M)_w)\simeq M^w$.
\end{enumerate}
\end{lem}
\begin{proof}
We may assume $M^w = R_w(n)$ for some $n\in \Z$.
Then $D(R_w(n))\simeq R_w(-n)$.
Since $D(M^w)\simeq D(M)_w$, we get the first part.
We also have $D(D(M^w))\simeq M^w$.
Therefore we have $D(D(M)_w)\simeq D(D(M^w))\simeq M^w$.
\end{proof}

\begin{lem}
We have $D(B_s\otimes M)\simeq B_s\otimes D(M)$ for any $M\in \mathcal{C}$ and $s\in S$.
In particular we have $D(B_{\underline{x}})\simeq B_{\underline{x}}$ for any $\underline{x}\in S^l$ and $D$ preserves $\BSbimod$ and $\Sbimod$.
\end{lem}
\begin{proof}
We regard $B_s\otimes (\cdot) = R\otimes_{R^s}(\cdot)$.

Take $\delta\in V$ such that $\langle\alpha_s^\vee,\delta\rangle = 1$.
For $\varphi\in D(B_s\otimes M)$, define $\varphi_1,\varphi_2\colon M\to R$ by $\varphi_1(m) = \varphi(1\otimes m)$ and $\varphi_2(m) = \varphi(\delta\otimes m)$.
It is clear that these are right $R$-homomorphisms, hence defines elements of $D(M)$.
Set $\Phi(\varphi) = 1\otimes \varphi_2 - s(\delta)\otimes \varphi_1\in B_s\otimes D(M)$.

We prove that $\Phi\colon D(B_s\otimes M)\to B_s\otimes D(M)$ is an $R$-bimodule homomorphism.
Clearly we have $(\varphi f)_1 = \varphi_1 f$ and $(\varphi f)_2 = \varphi_2 f$, hence $\Phi$ is a right $R$-module homomorphism.
If $f\in R^s$, then $(f\varphi)_1(m) = \varphi(f\otimes m) = \varphi(1\otimes fm) = (f(\varphi)_1)(m)$.
Hence $(f\varphi)_1 = f\varphi_1$ and similarly we have $(f\varphi)_2 = f\varphi_2$.
Therefore $\Phi(f\varphi) = 1\otimes f\varphi_2 - s(\delta)\otimes f\varphi_1 = f(1\otimes \varphi_2 - s(\delta)\otimes \varphi_1) = f\Phi(\varphi)$.
We prove $\Phi(\delta\varphi) = \delta\Phi(\varphi)$.
We have $((\delta\varphi)_2)(m) = \varphi(\delta^2\otimes m) = \varphi(\delta(\delta + s(\delta))\otimes m - \delta s(\delta)\otimes m) = \varphi(\delta\otimes (\delta + s(\delta))m) - \varphi(1\otimes \delta s(\delta)m) = \varphi_2((\delta + s(\delta))m) - \varphi_1(\delta s(\delta)m) = ((\delta + s(\delta))\varphi_2 - \delta s(\delta)\varphi_1)(m)$.
We also have $(\delta\varphi)_1 = \varphi_2$.
Hence $\Phi(\delta \varphi) = 1\otimes ((\delta + s(\delta))\varphi_2 - \delta s(\delta)\varphi_1) - s(\delta)\otimes \varphi_2 = (\delta + s(\delta))\otimes \varphi_2 - \delta s(\delta)\otimes \varphi_1 - s(\delta)\otimes \varphi_2 = \delta(1\otimes \varphi_2 - s(\delta)\otimes\varphi_1) = \delta\Phi(\varphi)$.
By Lemma~\ref{lem:basis over R^s}, $\Phi$ is a left $R$-module, hence $R$-bimodule homomorphism.

We prove that $\Phi$ preserves the decomposition over $Q$.
Let $\varphi\in D(B_s\otimes M)_Q^w$.
Since $B_s\otimes M_Q^y\subset (B_s\otimes M)_Q^y + (B_s\otimes M)_Q^{sy}$ by Lemma~\ref{lem:support of B_s otimes}, $\varphi_1,\varphi_2$ is zero on $M_Q^y$ if $y\ne w,sw$.
Hence $\supp_W(\varphi_1),\supp_W(\varphi_2)\subset \{w,sw\}$.
Therefore $\supp_W(\Phi(\varphi))\subset \{w,sw\}$.

The homomorphism $\varphi$ is zero on $(B_s\otimes M)^{sw}_Q$.
By Lemma~\ref{lem:stalk of B_s otimes M}, $\varphi(\delta \otimes m - 1\otimes s(\delta) m) = 0$ for $m\in M_Q^{sw}$.
Namely we have $(\varphi_2)_{sw} = s(\delta)(\varphi_1)_{sw}$.
Similarly we have $(\varphi_2)_w = \delta(\varphi_1)_w$.
The image of $\Phi(\varphi)$ in $(B_s\otimes D(M))_Q^{sw} = D(M)_Q^{sw}\oplus D(M)_Q^{w}$ is $((\varphi_2)_{sw} - s(\delta)(\varphi_1)_{sw},(\varphi_2)_w - \delta(\varphi_1)_w)$ and this is zero.
Therefore $\supp_W(\Phi(\varphi))\subset \{w\}$.

Since $B_s = 1\otimes R\oplus \delta_s\otimes R$, we have $D(B_s\otimes M) = D(1\otimes M\oplus \delta_s\otimes M) = D(1\otimes M)\oplus D(\delta_s\otimes M)$.
Then $\Phi$ sends $D(1\otimes M)$ (resp.\ $D(\delta_s\otimes M)$) to $s(\delta)\otimes D(M)$ (resp.\ $1\otimes D(M)$).
Since $B_s\otimes D(M) = s(\delta)\otimes D(M)\oplus 1\otimes D(M)$, we proved that $\Phi$ is an isomorphism.
\end{proof}

\section{Light leaves}\label{sec:Light leaves}
\subsection{Notation}
Let $\underline{x} \in S^l$ and $\boldsymbol{e} = (\boldsymbol{e}_1,\dots,\boldsymbol{e}_l)\in \{0,1\}^l$.
We set $\underline{x}^{\boldsymbol{e}} = s_1^{\boldsymbol{e}_1}\dots s_l^{\boldsymbol{e}_l}\in W$.
Let $x_0 = 1,x_1 = s_1^{\boldsymbol{e}_1},x_2 = s_1^{\boldsymbol{e}_1}s_2^{\boldsymbol{e}_2},\dots,x_l = s_1^{\boldsymbol{e}_1}\dotsm s_l^{\boldsymbol{e}_l} = \underline{x}^{\boldsymbol{e}}$.
Using this sequence, we add a label to $\boldsymbol{e}$ at each index.
We assign $U$ to the index $i$ if $x_{i - 1}s_i > x_{i - 1}$ and $D$ otherwise.
The \emph{defect} $d(\boldsymbol{e})$ of $\boldsymbol{e}$ is defined by
\[
d(\boldsymbol{e}) = \#\{i\mid \text{the label is $U$ at $i$}, \boldsymbol{e}_i = 0\} - \#\{i\mid \text{the label is $D$ at $i$}, \boldsymbol{e}_i = 0\}.
\]
Of course, this number depends on $\underline{x}$, not only on $\boldsymbol{e}$.
\begin{lem}[{\cite[Lemma~2.7]{MR3555156}}]\label{lem:defect formula}
We have $p_{\underline{x}}^w(v) = \sum_{\underline{x}^{\boldsymbol{e}} = w}v^{d(\boldsymbol{e})}$.
\end{lem}

\subsection{Assumption}
For $\underline{x}\in S^l$, define $u_{\underline{x}}\in B_{\underline{x}}$ by $u_{\underline{x}} = (1\otimes 1)\otimes (1\otimes 1)\otimes\dotsm\otimes (1\otimes 1)$.
In the rest of this paper, we assume the following.
\begin{assump}\label{assump:existecne of 2-coloerd map}
For any $s,t\in S$ such that $s\ne t$ and the order $m$ of $st$ is finite, we have the following.
Set $\underline{x} = (s,t,\dots)\in S^m$ and $\underline{y} = (t,s,\dots)\in S^m$.
Then there exists a zero-degree morphism $B_{\underline{x}}\to B_{\underline{y}}$ in $\mathcal{C}$ which sends $u_{\underline{x}}$ to $u_{\underline{y}}$.
\end{assump}

We have the following sufficient condition.

\begin{lem}
If the following condition holds for any $s,t\in S$ such that $st$ has the finite order, Assumption \ref{assump:existecne of 2-coloerd map} holds.
\begin{itemize}
\item Let $T'$ be the refractions in the group $\langle s,t\rangle$ generated by $\{s,t\}$.
For any $t_1,t_2\in T'$ with $t_1\ne t_2$ there exists $v\in V$ such that $\langle v,\alpha_{t_1}^\vee \rangle = 0$, $\langle v,\alpha_{t_2}^\vee\rangle = 1$.
\end{itemize}
\end{lem}

Assume that $\Coef$ is a field and the representation $V$ restricted to $\langle s,t\rangle$ is reflection faithful.
Then Soergel's categorification theorem holds for $\langle s,t\rangle$.
Let $w_0$ be the longest element in $\langle s,t\rangle$.
Then there exists a corresponding indecomposable object $B(w_0)$ and the object $B(w_0)$ is a direct summand of $B_{\underline{x}}$ and $B_{\underline{y}}$.
Therefore we have a map $B_{\underline{x}}\twoheadrightarrow B(w_0)\hookrightarrow B_{\underline{y}}$ and after suitable normalization this gives a desired map of the assumption.
The proof of the above lemma basically follows this tactics.
In the proof below, we follow Soergel's argument in \cite{MR2329762}.

In the rest of this subsection, let $s,t\in S$ be as in the lemma and set $W' = \langle s,t\rangle$.
Let $T'$ be the set of reflections in $W'$.
\begin{lem}
The representation $V$ restricted to $W'$ is faithful.
\end{lem}
\begin{proof}
Let $w\in W'$ and assume that $w$ acts trivially on $V$.
First assume that the length of $w$ is odd.
If $\ell(w) > 1$, then for $s'  = s$ or $t$, we have $\ell(sws) < \ell(w)$ and $sws$ is also acts trivially on $V$.
Therefore we conclude $s$ or $t$ acts trivially.
This contradicts to $\alpha_{s'}^\vee,\alpha_{s'}\ne 0$ for $s'\in \{s,t\}$.
Assume that $\ell(w)$ is even.
Then there exist $t_1,t_2\in T'$ such that $w = t_1t_2$.
If $w\ne 1$, then $t_1\ne t_2$.
Take $v\in V$ as in the assumption of the lemma.
Then $t_1(v) = v$ and $t_2(v) = v - \alpha_{t_2}\ne v$.
Therefore $t_1(v)\ne t_2(v)$.
Hence the action of $w$ is not trivial.
\end{proof}
Therefore, the decomposition $M^Q = \bigoplus_{x\in W'}M_x^Q$ is determined uniquely by the $R$-bimodule structure of $M$.
Hence it is sufficient to prove the existence of an $R$-bimodule homomorphism in the assumption.

We also note the following: if $t_1,t_2\in T'$ are different, then $\{\alpha_{t_1}^\vee,\alpha_{t_2}^\vee\}$ is linearly independent.
Namely, $\{\alpha_{t_1}^\vee,\alpha_{t_2}^\vee\}$ generates the two dimensional space over $\Frac(\Coef)$.
Note that this space is contained in $\Frac(\Coef)\alpha_s^\vee + \Frac(\Coef)\alpha_t^\vee$ and therefore it is equal.
Namely the two dimensional space does not depend on $t_1,t_2$.

For each $x\in W'$, set $\Gr(x) = \{(v,x^{-1}(v))\in V^*\times V^*\mid v\in V^*\}$ and for each $A\subset W'$, let $R(A)$ be the space of regular functions on $\bigcup_{x\in A}\Gr(x)$.
In other words, $R(A)$ is the image of $R\otimes_{\Coef} R$ in $R^{A}$ under the map $f\otimes g\mapsto (fx(g))_{x\in A}$, here we use an identification $\Gr(x)\simeq V^*$ via $(v,x^{-1}(v))\mapsto v$.
The embedding is an isomorphism after tensoring $Q$, namely we have $R(A)\otimes Q\simeq Q^A$ and this decomposition gives a structure of an object of $\mathcal{C}'$ to $R(A)$.
In this subsection, we consider an object of $\mathcal{C}'$ as graded right $R\otimes_{\Coef}R$-modules, rather than graded $R$-bimodules.

Let $R(A)^+$ be the image of $R\otimes R^s$ in $R(A)$.
\begin{lem}\label{lem:+ + tensor = all}
Assume that $As = A$.
Then $R(A)^{+}\otimes_{R^s}R\to R(A)$ given by $z\otimes f\mapsto z(1\otimes f)$ is an isomorphism.
\end{lem}
\begin{proof}
It is obvious that the map is surjective.
Therefore it is sufficient to prove that the map is injective.
Take $\delta_s\in V$ such that $\langle \alpha_s^\vee,\delta_s\rangle = 1$.
Then we have $R = R^s\oplus R^s\delta_s$.
Therefore it is sufficient to prove that $R(A)^{+}\cap R(A)^{+}(1\otimes \delta_s) = 0$.

Let $z_1,z_2\in R(A)^{+}$ and assume that $z_1 = z_2(1\otimes \delta_s)$.
Let $((z_i)_x)$ be the image of $z_i$ in $R^{A}$ for $i = 1,2$.
By $z_1 = z_2(1\otimes \delta_s)$, we have $(z_1)_x = (z_2)_{x}x(\delta_s)$.
By replacing $x$ with $xs$, we have $(z_1)_{xs} = (z_2)_{xs}xs(\delta_s)$.
Since $z_1,z_2$ are $s$-invariant, we have $(z_1)_{xs} = (z_1)_x$ and $(z_2)_{xs} = (z_2)_{x}$.
Therefore we get $(z_2)_{x}(x(\delta_s) - xs(\delta_s)) = 0$.
Hence $(z_2)_{x}x(\alpha_s) = 0$.
We get $(z_2)_{x} = 0$ for any $x\in A$ and this implies $z_2 = 0$.
\end{proof}

\begin{lem}\label{lem:Soergel's decomposition, rank 2}
Let $x\in W'$ and assume that $xs > x$.
Set $A = \{y\in W'\mid y\le x\}$.
Then $R(A)\otimes_{R}B_s\simeq R(A\cup As)(1)\oplus R(A\cap As)(-1)$.
\end{lem}
\begin{proof}
Since $xs > x$, we have $A\setminus As = \{x,xr\}$ for some $r\in T'\setminus\{s\}$.
Take $v_0\in V$ such that $\langle \alpha_r^\vee,v_0\rangle = 0$ and $\langle \alpha_s^\vee,v_0\rangle = 1$ and consider $z_0 = x(v_0)\otimes 1 - 1\otimes v_0$.
Note that we have the restriction map $R(A)\to R(A\cap As)$.
\begin{claim}
Let $f\in R(A)$.
Then we have $fz_0 = 0$ in $R(A)$ if and only if $f = 0$ in $R(A\cap As)$.
\end{claim}
Recall the embedding $R(A)\to R^A$.
Let $(f_y)_{y\in A}$ be the image of $f$.
Then the image of $fz_0$ is $(f_y(x(v_0) - y(v_0)))_{y\in A}$.
When $y \in \{x,xr\} = A\setminus As$, then $x(v_0) - y(v_0) = 0$ since $r(v_0) = v_0$ from the assumption on $v_0$.
Therefore if $f = 0$ in $R(A\cap As)$, then since $f_y = 0$ for any $A\cap As$, we have $f_y(x(v_0) - y(v_0)) = 0$ for any $y\in A$.
We get $fz_0 = 0$ in $R(A)$.

We prove $x(v_0) - y(v_0) \ne 0$ for any $y\in A\cap As$.
This give the reverse implication.
We assume that $y^{-1}x(v_0) = v_0$.
Then we also have $y^{-1}xr(v_0) = v_0$.
One of $y^{-1}x$ or $y^{-1}xr$ is in $T'$.
Let $r'$ be the element in $T'$.
Then $r(v_0) = r'(v_0) = v_0$ and since $y\ne x,rx$, we have $r\ne r'$.
Therefore $\langle \alpha_r^\vee,v_0\rangle = \langle \alpha_{r'}^\vee,v_0\rangle = 0$.
Hence $v_0$ is orthogonal to $\Frac(\Coef)\alpha_r + \Frac(\Coef)\alpha_{r'}$.
Recall that this space does not depend on $r,r'$.
Hence $v_0$ is orthogonal to $\Frac(\Coef)\alpha_r + \Frac(\Coef)\alpha_s$.
This contradicts to the assumption $\langle \alpha_s,v_0\rangle = 1$.
We get the claim.

By the claim, we have an isomorphism $R(A\cap As)(-2)\simeq R(A)z_0$.
Let $M$ be the sub $(R,R^s)$-module of $R(A)$ generated by $z_0$.
Then the above isomorphism implies $R(A\cap As)^{+}(-2)\simeq M$.

Next consider the restriction map $R(A\cup As)\to R(A)$.
Since $A\cup As$ is the right $s$-invariant, we have an action of $s$ on $R(A\cup As)$.
Let $R(A\cup As)^{s = 1}$ be the $s$-fixed part.
Then $R(A\cup As)^{+}\subset R(A\cup As)^{s = 1}$ and the above restriction map gives an injective map $R(A\cup As)^{s = 1}\hookrightarrow R(A)$.
Let $N$ be the image of $R(A\cup As)^{+}\hookrightarrow R(A)$.
We have $N\simeq R(A\cup As)^{+}$.

To prove $R(A)\otimes_{R}B_s\simeq R(A\cup As)(1)\oplus R(A\cap As)(-1)$, it is sufficient to prove $R(A) = M\oplus N$ by Lemma~\ref{lem:+ + tensor = all}.
The subspace $M + N$ is the image of $(R\otimes R^s)z_0 + R\otimes R^s$.
Since $z_0 = x(v_0)\otimes 1 + 1 \otimes v_0$, we have $(R\otimes R^s)z_0 + R\otimes R^s = R\otimes R^sv_0 + R\otimes R^s$.
We have $R = R^s + R^sv_0$ since $\langle \alpha_s^\vee,v_0\rangle = 1$.
Hence $(R\otimes R^s)z_0 + R\otimes R^s = R\otimes R$.
Therefore $M + N = R(A)$.

Let $m\in M$ and $n\in N$.
We can write $m = fz_0$ for some $f\in \Ima(R\otimes R^s\to R(A))$.
We assume $m = n$.
Consider the image of $f$ in $R(A\cap As)$ and let $(f_y)_{y\in A\cap As}$ be the image in $R^{A\cap As}$.
We also denote the image of $n$ in $R^{A\cap As}$ by $(n_y)_{y\in A\cap As}$.
Then for each $y\in A\cap As$, we have $f_y(x(v_0) - y(v_0)) = n_{y}$.
We also have $f_{ys}(x(v_0) - ys(v_0)) = n_{ys}$.
Since $f$ and $n$ are $s$-invariant, $f_{ys} = f_{y}$, $n_{ys} = n_y$.
Therefore $f_{y}(ys(v_0) - y(v_0)) = 0$.
Namely we have $f_{y}y(\alpha_s) = 0$.
We get $f_y = 0$ and hence $fz_0\in M$ is zero.
\end{proof}
We have proved $R(A)\otimes_{R}B_s\simeq R(A\cup As)(1)\oplus R(A\cap As)(-1)$.
For $A = \{y\in W'\mid y\le x\}$ with $xs > x$, we have $A\cup As = \{y\in W'\mid y\le xs\}$.
We set $R(\le x) = R(\{y\in W'\mid y\le x\})$ for $x\in W'$.
Then $R(\le xs)(1)$ is a direct summand of $R(\le x)\otimes_{R}B_s$ if $xs > x$.
Replacing $s$ with $t$, $R(\le yt)(1)$ is a direct summand of $R(\le y)\otimes_{R}B_s$ if $yt > y$.
Therefore we have the following.
Let $\underline{w} \in \{s,t\}^{\ell(w)}$ be a reduced expression of $w\in W$.
Then $B_{\underline{w}}$ has a direct summand $R(\le w)(\ell(w))$.
Let $w_0\in W'$ be the longest element.
Then for $\underline{x}$ and $\underline{y}$ as in Assumption~\ref{assump:existecne of 2-coloerd map}, $R(\le w_0)(\ell(w_0))$ is a direct summand of both $B_{\underline{x}}$ and $B_{\underline{y}}$.
Therefore we get a non-zero homomorphism $\Phi\colon B_{\underline{x}}\twoheadrightarrow R(\le w_0)(\ell(w_0))\hookrightarrow B_{\underline{y}}$.

\begin{lem}
The homomorphism $\Phi$ induces an isomorphism $B_{\underline{x}}^{w_0}\simeq B_{\underline{y}}^{w_0}$.
\end{lem}
\begin{proof}
By induction, we can prove that $B_{\underline{w}}\simeq R(\le w)(\ell(w))\oplus M$ with $\supp_{W'}(M)\subset \{y \in W'\mid y < w\}$ for any $w\in W'$ and a reduced expression $\underline{w}$ of $w$.
Therefore $B_{\underline{x}} = R(\le w_0)(\ell(w_0))\oplus M$ with $\supp_{W'}(M)\subset W'\setminus\{w_0\}$.
Hence $B_{\underline{x}}^{w_0} = R(\le w_0)(\ell(w_0))^{w_0}$ since $M^{w_0} = 0$.
Therefore the projection $B_{\underline{x}}\twoheadrightarrow R(\le w_0)(\ell(w_0))$ induces an isomorphism $B_{\underline{x}}^{w_0}\simeq R(\le w_0)(\ell(w_0))^{w_0}$.
Similarly the embedding $R(\le w_0)(\ell(w_0))\hookrightarrow B_{\underline{y}}$ induces an isomorphism $R(\le w_0)(\ell(w_0))^{w_0}\hookrightarrow B_{\underline{y}}^{w_0}$.
We get the lemma by the construction of $\Phi$.
\end{proof}

Define $\varphi_{\underline{x}}\colon B_{\underline{x}}\simeq R\otimes_{R^s}R\otimes_{R^t}\otimes\cdots\otimes R\to R(-\ell(w_0))$ by $f_0\otimes f_1\otimes\dots\otimes f_{\ell(w_0)}\mapsto f_0(s(f_1))(st(f_2))\dotsm (st\dotsb (f_l))$.
(Recall that $\underline{x} = (s,t,\ldots)$.)
Then this satisfies $\varphi_{\underline{x}}(uf) = \varphi_{\underline{x}}(w_0(f)u)$ for any $u\in B_{\underline{x}}$ and $f\in R$.
Therefore it factors through $B_{\underline{x}}\to B_{\underline{x}}^{w_0}$.
We get $\overline{\varphi}_{\underline{x}}\colon B_{\underline{x}}^{w_0}\to R(-\ell(w_0))$.
This is surjective.
On the other hand, after tensoring $Q$, both are one-dimensional $Q$-vector space by Lemma~\ref{lem:Q-dimension of BS module}.
Therefore it is an isomorphism.
In particular, $\overline{\varphi}_{\underline{x}}$ is injective, hence it is an isomorphism.

Note that the $(-\ell(w))$-th degree part of $B_{\underline{x}}$ (resp.\ $B_{\underline{y}}$) is free of rank one $\Coef$-module with a basis $u_{\underline{x}}$ (resp.\ $u_{\underline{y}}$).
Therefore $\Phi(u_{\underline{x}}) = cu_{\underline{y}}$ for some $c\in\Coef$.
Let $\overline{u}_{\underline{x}}\in B_{\underline{x}}^{w_0}$ (resp.\ $\overline{u}_{\underline{y}}\in B_{\underline{y}}^{w_0}$) be the image of $u_{\underline{x}}$ (resp.\ $u_{\underline{y}}$).
Since $\varphi_{\underline{x}}(u_{\underline{x}}) = 1$, we have the following commutative diagram:
\[
\begin{tikzcd}
\overline{u}_{\underline{x}}\arrow[rrr,mapsto,bend left]\arrow[d,mapsto] & B_{\underline{x}}^{w_0}\arrow[l,phantom,"\in"]\arrow[r,"\sim"]\arrow[d,"\wr"] & B_{\underline{y}}^{w_0}\arrow[d,"\wr"]\arrow[r,phantom,"\ni"] & c\overline{u}_{\underline{y}}\arrow[d,mapsto] \\
1& R(-\ell(w_0))\arrow[l,phantom,"\in"]\arrow[r] & R(-\ell(w_0))\arrow[r,phantom,"\ni"] & c.
\end{tikzcd}
\]
We have $c\in \Coef^\times$ and $c^{-1}\Phi$ gives a desired homomorphism.

\subsection{Light leaves}
We recall the definition of light leaves \cite{MR2441994} following notation of \cite{MR3555156}.
We need one more notation.

Let $w\in W$ and let $\underline{x},\underline{y}\in S^{\ell(w)}$ be two reduced expressions of $w$.
Then by a fundamental property of a Coxeter system, there exists a sequence $\underline{x}^0 = \underline{x},\underline{x}^1,\dots,\underline{x}^r = \underline{y}$ such that each $\underline{x}^i$ and $\underline{x}^{i + 1}$ only differs with a single braid relation.
In each step, we can attach a homomorphism $B_{\underline{x}^i}\to B_{\underline{x}^{i + 1}}$ using the homomorphism in Assumption~\ref{assump:existecne of 2-coloerd map}.
We write the composite $\rex$.
Note that $\rex$ sends $u_{\underline{x}}$ to $u_{\underline{y}}$.
Of course, this homomorphism is not unique.
We fix it for any such two reduced expressions.
See \cite[4.2]{MR3555156} for the details.

For the definition of light leaves, we use the following maps.
Let $s\in S$.
First set $\partial_s(f) = (f - s(f))/\alpha_s$.
Put
\begin{align*}
m^s\colon B_s\to R &\qquad f_1\otimes f_2\mapsto f_1f_2 && \text{degree $1$},\\
i_0^s\colon B_s\otimes B_s\to B_s &\qquad f_1\otimes f_2\otimes f_3 \mapsto f_1\partial_s(f_2)\otimes f_3 && \text{degree $-1$},\\
i_1^s\colon B_s\otimes B_s \to R &\qquad f_1\otimes f_2\otimes f_3 \mapsto f_1\partial_s(f_2)f_3 && \text{degree $0$}.
\end{align*}
Here $f_1,f_2,f_3\in R$ and, in the definition of $i_0^s,i_1^s$, we regard $B_s\otimes B_s = R\otimes_{R^s}R\otimes_{R^s}R$.
Since $V$ is a faithful $\{e,s\}$-representation, these are homomorphisms in $\mathcal{C}$.

\begin{defn}\label{defn:light leaves}
Let $\underline{x} = (s_1,\dots,s_l)\in S^l$, $\boldsymbol{e} = (\boldsymbol{e}_1,\dots,\boldsymbol{e}_l) \in \{0,1\}^l$.
Then we define a light leaf $\LL_{\underline{x},\boldsymbol{e}}$ which is a morphism from $B_{\underline{x}}$ to $B_{\underline{w}}$ in $\mathcal{C}$ where $\underline{w}$ is a fixed reduced expression of $w = \underline{x}^{\boldsymbol{e}}$.
Let $\underline{x}_{\le k} = (s_1,\dots,s_k)$, $\boldsymbol{e}_{\le k} = (\boldsymbol{e}_1,\dots,\boldsymbol{e}_k)$ and $w_k = \underline{x}_{\le k}^{\boldsymbol{e}_{\le k}}$.
We fix a reduced expression $\underline{w}_k$ of $w_k$ and define $\LL_k\colon B_{\underline{x}_{\le k}}\to B_{\underline{w}_k}$ inductively as follows.

\begin{enumerate}
\item[($U0$)] $\boldsymbol{e}_k = 0$ and $w_{k - 1}s_k > w_{k - 1}$.
\[
B_{\underline{x}_{\le k - 1}}\otimes B_{s_k}\xrightarrow{\LL_{k - 1}\otimes \Id_{B_{s_k}}}B_{\underline{w}_{k - 1}}\otimes B_{s_k}\xrightarrow{\Id_{B_{\underline{w}_{k - 1}}}\otimes m^{s_k}} B_{\underline{w}_{k - 1}}.
\]
\item[($U1$)] $\boldsymbol{e}_k = 1$ and $w_{k - 1}s_k > w_{k - 1}$.
\[
B_{\underline{x}_{k - 1}}\otimes B_{s_k}\xrightarrow{\LL_{k - 1}\otimes\Id_{B_{s_k}}} B_{\underline{w}_{k - 1}}\otimes B_{s_k} \xrightarrow{\rex}B_{\underline{w}_k}.
\]

\item[($D0$)] $\boldsymbol{e}_k = 0$ and $w_{k - 1}s_k < w_{k - 1}$.
Let $(t_1,\dots,t_{p - 1},s_k)$ be a reduced expression of $w_{k - 1}$ ending with $s_k$.
\begin{multline*}
B_{\underline{x}_{\le k - 1}}\otimes B_{s_k}\xrightarrow{\LL_{k - 1}\otimes \Id_{B_{s_k}}}B_{\underline{w}_{k - 1}}\otimes B_{s_k}\xrightarrow{\rex\otimes\Id_{B_{s_k}}}B_{(t_1,\dots ,t_{p - 1},s_k)}\otimes B_{s_k}\\
\xrightarrow{\Id_{B_{(t_1,\dots ,t_{p - 1})}}\otimes i_0^{s_k}}B_{(t_1,\dots ,t_{p - 1},s_k)}\xrightarrow{\rex}B_{\underline{w}_k}.
\end{multline*}

\item[($D1$)] $\boldsymbol{e}_k = 1$ and $w_{k - 1}s_k < w_{k - 1}$.
Let $(t_1,\dots,t_{p - 1},s_k)$ be a reduced expression of $w_{k - 1}$ ending with $s_k$.
\begin{multline*}
B_{\underline{x}_{\le k - 1}}\otimes B_{s_k}\xrightarrow{\LL_{k - 1}\otimes \Id_{B_{s_k}}}B_{\underline{w}_{k - 1}}\otimes B_{s_k}\xrightarrow{\rex\otimes \Id_{B_{s_k}}}B_{(t_1,\dots, t_{p - 1},s_k)}\otimes B_{s_k}\\
\xrightarrow{\Id_{B_{(t_1,\dots ,t_{p - 1})}}\otimes i_1^{s_k}}B_{(t_1,\dots ,t_{p - 1})}\xrightarrow{\rex}B_{\underline{w}_k}.
\end{multline*}
\end{enumerate}
Finally we put $\LL_{\underline{x},\boldsymbol{e}} = \LL_l$.
By the construction, the degree of $\LL_{\underline{x},\boldsymbol{e}}$ is $d(\boldsymbol{e})$.
\end{defn}

We fix a sequence $\underline{x} = (s_1,\dots,s_l)$ in this subsection.
\begin{lem}
Let $\boldsymbol{e},\boldsymbol{f}\in \{0,1\}^l$ such that $\underline{x}^{\boldsymbol{e}} = \underline{x}^{\boldsymbol{f}}$.
Assume that the labels of $\boldsymbol{e}$ and $\boldsymbol{f}$ at $i$ are the same for all $i = 1,\dots,l$.
Then we have $\boldsymbol{e} = \boldsymbol{f}$.
\end{lem}
\begin{proof}
We prove $\boldsymbol{e}_i = \boldsymbol{f}_i$ by backward induction on $i$.
Assume that we have $\boldsymbol{e}_j = \boldsymbol{f}_j$ for any $j > i$.
Since $s_1^{\boldsymbol{e}_1}\dotsm s_l^{\boldsymbol{e}_l} = s_1^{\boldsymbol{f}_1}\dotsm s_l^{\boldsymbol{f}_l}$ by the assumption and $s_{i + 1}^{\boldsymbol{e}_{i + 1}}\dotsm s_l^{\boldsymbol{e}_l} = s_{i + 1}^{\boldsymbol{f}_{i + 1}}\dotsm s_l^{\boldsymbol{f}_l}$ by inductive hypothesis, we have $s_1^{\boldsymbol{e}_1}\dotsm s_i^{\boldsymbol{e}_i} = s_1^{\boldsymbol{f}_1}\dotsm s_i^{\boldsymbol{f}_i}$.

Assume that the label of $\boldsymbol{e}$ at $i$ (hence that of $\boldsymbol{f}$ at $i$) is $U$, $\boldsymbol{e}_i = 1$, $\boldsymbol{f}_i = 0$.
Set $w = s_1^{\boldsymbol{e}_1}\dotsm s_{i - 1}^{\boldsymbol{e}_{i - 1}}$.
Then $s_1^{\boldsymbol{f}_1}\dotsm s_{i - 1}^{\boldsymbol{f}_{i - 1}} = ws_i$ since $s_1^{\boldsymbol{e}_1}\dotsm s_{i}^{\boldsymbol{e}_{i}} = s_1^{\boldsymbol{f}_1}\dotsm s_{i}^{\boldsymbol{f}_{i}}$.
Then since the label of $\boldsymbol{e}$ at $i$ is $U$, we have $ws_i > w$ and since the label of $\boldsymbol{f}$ at $i$ is also $U$, we have $(ws_i)s_i > ws_i$.
This is a contradiction.
We also have a contradiction for other cases.
Hence we have $\boldsymbol{e}_i = \boldsymbol{f}_i$.
\end{proof}
Let $w\in W$.
Using this lemma, we can define a total order ${<} = {<_{\underline{x},w}}$ on $\{\boldsymbol{e}\in S^l\mid \underline{x}^{\boldsymbol{e}} = w\}$ as follows: $\boldsymbol{f} < \boldsymbol{e}$ if and only if there exists $i$ such that
\begin{itemize}
\item the labels of $\boldsymbol{e}$ and $\boldsymbol{f}$ are the same at any $j < i$.
\item the label of $\boldsymbol{e}$ is $D$ at $i$.
\item the label of $\boldsymbol{f}$ is $U$ at $i$.
\end{itemize}

Fix $\delta_s\in V$ such that $\langle \alpha_s^\vee,\delta_s\rangle = 1$ for each $s\in S$.
For $\boldsymbol{e} \in \{0,1\}^l$, we define $b_{\underline{x},\boldsymbol{e}}\in B_{\underline{x}}$ by $b_{\underline{x},\boldsymbol{e}} = b_1\otimes\dotsm \otimes b_l$ where $b_i\in B_{s_i}$ is defined by
\[
b_i = 
\begin{cases}
1\otimes 1 & (\text{The label of $\boldsymbol{e}$ at $i$ is $U$}),\\
\delta_{s_i} \otimes 1 & (\text{The label of $\boldsymbol{e}$ at $i$ is $D$}).
\end{cases}
\]
Then we have
\begin{prop}\label{prop:dual basis of light leaves}
Let $w\in W$, $\boldsymbol{e},\boldsymbol{f}\in \{0,1\}^l$ and assume that $\underline{x}^{\boldsymbol{e}} = \underline{x}^{\boldsymbol{f}} = w$.
Fix a reduced expression $\underline{w}$ of $w$.
Then we have
\[
\LL_{\underline{x},\boldsymbol{e}}(b_{\underline{x},\boldsymbol{f}}) = 
\begin{cases}
u_{\underline{w}} & (\boldsymbol{f} = \boldsymbol{e}),\\
0 & (\boldsymbol{f} < \boldsymbol{e}).
\end{cases}
\]
In particular, $\{\LL_{\underline{x},\boldsymbol{e}}\mid \underline{x}^{\boldsymbol{e}} = w\}$ is linearly independent.
\end{prop}
\begin{proof}
Let $\underline{x}_{\le k}$, $\boldsymbol{e}_{\le k}$, $\underline{w}_k$ as in Definition~\ref{defn:light leaves} and $\boldsymbol{f}_{\le k}$ similarly.
We prove that if the labels of $\boldsymbol{e}$ and $\boldsymbol{f}$ are the same at $i \le k$, then $\LL_k(b_{\underline{x}_{\le k},\boldsymbol{f}_{\le k}}) = u_{\underline{w}_k}$ by induction on $k$.

Assume that the label of $\boldsymbol{e}$ at $k$ is $U$.
Then we have $b_{\underline{x}_{\le k},\boldsymbol{f}_{\le k}} = b_{\underline{x}_{\le k - 1},\boldsymbol{f}_{\le k - 1}}\otimes (1\otimes 1)$.

If $\boldsymbol{e}_k = 0$, then we have 
\[
\LL_k(b_{\underline{x}_{\le k},\boldsymbol{f}_{\le k}}) = \LL_{k - 1}(b_{\underline{x}_{\le k - 1},\boldsymbol{f}_{\le k - 1}})\otimes m^{s_k}(1\otimes 1) = u_{\underline{w}_{k - 1}}\otimes m^{s_k}(1\otimes 1) = u_{\underline{w}_{k}}.
\]

If $\boldsymbol{e}_k = 1$, then 
\[
\LL_k(b_{\underline{x}_{\le k},\boldsymbol{f}_{\le k}}) = \rex(\LL_{k - 1}(b_{\underline{x}_{\le k - 1},\boldsymbol{f}_{\le k - 1}})\otimes (1\otimes 1)) = \rex(u_{\underline{w}_{k - 1}}\otimes (1\otimes 1)) = u_{\underline{w}_k}.
\]

Assume that the label of $\boldsymbol{e}$ at $k$ is $D$.
Then we have $b_{\underline{x}_{\le k},\boldsymbol{f}_{\le k}} = b_{\underline{x}_{\le k - 1},\boldsymbol{f}_{\le k - 1}}\otimes (\delta_{s_k}\otimes 1)$.
Let $(t_1,\dots,t_{p - 1},s_k)$ be a reduced expression of $w_{k - 1}$.
If $\boldsymbol{e}_k = 0$, then we have 
\begin{align*}
\LL_k(b_{\underline{x}_{\le k},\boldsymbol{f}_{\le k}}) & = \rex((\Id\otimes i_0^{s_k})(\rex(\LL_{k - 1}(b_{\underline{x}_{\le k - 1},\boldsymbol{f}_{\le k - 1}}))\otimes (\delta_{s_k}\otimes 1)))\\
& = \rex((\Id\otimes i_0^{s_k})(u_{(t_1,\dots,t_{p - 1},s_k)}\otimes (\delta_{s_k}\otimes 1)))\\
& = \rex(u_{(t_1,\dots,t_{p - 1})}\otimes i_0^{s_k}(1\otimes \delta_{s_k}\otimes 1))\\
& = \rex(u_{(t_1,\dots,t_{p - 1})}\otimes (\partial_{s_k}(\delta_{s_k})\otimes 1))\\
& = \rex(u_{(t_1,\dots,t_{p - 1},s_k)})\\
& = u_{\underline{w}_k}.
\end{align*}

If $\boldsymbol{e}_k = 1$, then we have 
\begin{align*}
\LL_k(b_{\underline{x}_{\le k},\boldsymbol{f}_{\le k}}) & = \rex((\Id\otimes i_1^{s_k})(\rex(\LL_{k - 1}(b_{\underline{x}_{\le k - 1},\boldsymbol{f}_{\le k - 1}}))\otimes (\delta_{s_k}\otimes 1)))\\
& = \rex((\Id\otimes i_1^{s_k})(u_{(t_1,\dots,t_{p - 1},s_k)}\otimes (\delta_{s_k}\otimes 1)))\\
& = \rex(u_{(t_1,\dots,t_{p - 1})}\otimes i_1^{s_k}(1\otimes \delta_{s_k}\otimes 1))\\
& = \rex(u_{(t_1,\dots,t_{p - 1})}\partial_{s_k}(\delta_{s_k}))\\
& = \rex(u_{(t_1,\dots,t_{p - 1})})\\
& = u_{\underline{w}_{k}}.
\end{align*}
This is the end of the induction.
In particular, we have $\LL_{\underline{x},\boldsymbol{e}}(b_{\underline{x},\boldsymbol{e}}) = u_{\underline{w}}$.

Assume that $\boldsymbol{f} < \boldsymbol{e}$ and take $k$ such that the labels of $\boldsymbol{e}$ and $\boldsymbol{f}$ are the same at any $i < k$ and the label of $\boldsymbol{e}$ (resp.\ $\boldsymbol{f}$) at $k$ is $D$ (resp.\ $U$).
Then we have $b_{\underline{x}_{\le k},\boldsymbol{f}_{\le k}} = b_{\underline{x}_{\le k - 1},\boldsymbol{f}_{\le k - 1}}\otimes (1\otimes 1)$.
Let $(t_1,\dots,t_{p - 1},s_k)$ be a reduced expression of $w_{k - 1}$.
If $\boldsymbol{e}_k = 0$, then we have 
\begin{align*}
\LL_k(b_{\underline{x}_{\le k},\boldsymbol{f}_{\le k}}) & = \rex((\Id\otimes i_0^{s_k})(\rex(\LL_{k - 1}(b_{\underline{x}_{\le k - 1},\boldsymbol{f}_{\le k - 1}}))\otimes (1\otimes 1)))\\
& = \rex(u_{(t_1,\dots,t_{p - 1})}\otimes i_0^{s_k}(1\otimes 1\otimes 1))\\
& = \rex(u_{(t_1,\dots,t_{p - 1})}\otimes (\partial_{s_k}(1)\otimes 1))\\
& = 0.
\end{align*}
If $\boldsymbol{e}_k = 1$, then we have 
\begin{align*}
\LL_k(b_{\underline{x}_{\le k},\boldsymbol{f}_{\le k}}) & = \rex((\Id\otimes i_1^{s_k})(\rex(\LL_{k - 1}(b_{\underline{x}_{\le k - 1},\boldsymbol{f}_{\le k - 1}}))\otimes (1\otimes 1)))\\
& = \rex(u_{(t_1,\dots,t_{p - 1})}\otimes i_1^{s_k}(1\otimes 1\otimes 1))\\
& = \rex(u_{(t_1,\dots,t_{p - 1})}\partial_{s_k}(1))\\
& = 0.
\end{align*}

These calculations imply $\LL_{\underline{x},\boldsymbol{e}}(b_{\underline{x},\boldsymbol{f}}) = 0$.
\end{proof}

\begin{rem}
Since the degree of $\LL_{\underline{x},\boldsymbol{e}}$ (resp.\ $u_{\underline{w}}$) is $d(\boldsymbol{e})$ (resp.\ $-\ell(w)$), $\LL_{\underline{x},\boldsymbol{e}}(b_{\underline{x},\boldsymbol{e}}) = u_{\underline{w}}$ implies $\deg(b_{\underline{x},\boldsymbol{e}}) = -d(\boldsymbol{e})-\ell(w)$.
\end{rem}

\subsection{A basis of $B_{\underline{x}}^w$}
Let $w\in W$ with a reduced expression $\underline{w}$ and $\underline{x}\in S^l$, $\boldsymbol{e}\in \{0,1\}^l$ such that $\underline{x}^{\boldsymbol{e}} = w$.
Let $b_{\underline{x},\boldsymbol{e}}^w$ be the image of $b_{\underline{x},\boldsymbol{e}}\in B_{\underline{x}}$ in $B_{\underline{x}}^w$.

Let $\underline{w} = (t_1,\dots,t_r)$.
We define a morphism $\varphi_{\underline{w}}\colon B_{\underline{w}}\to R_w$ in $\mathcal{C}$ by $\varphi_{\underline{w}}(f_0\otimes\dots\otimes f_r) = f_0(t_1(f_1))\dotsm (t_1\dotsm t_r(f_r))$ here $f_0,\dots,f_r\in R$ and we identify $B_{\underline{w}} = R\otimes_{R^{t_1}}R\otimes_{R^{t_2}}\dotsm\otimes_{R^{t_l}}R$.
\begin{thm}\label{thm:graded rank of B_x^w}
Fix $\underline{x}$, $w$ and $\underline{w}$.
\begin{enumerate}
\item The left $R$-module $B_{\underline{x}}^w$ has a basis $\{b^w_{\underline{x},\boldsymbol{e}}\mid \underline{x}^{\boldsymbol{e}} = w\}$.
\item The left $R$-module $B_{\underline{x}}^w$ is graded free and its graded rank $\grk(B_{\underline{x}}^w)$ is given by $\sum_{\underline{x}^{\boldsymbol{e}} = w}v^{d(\boldsymbol{e}) + \ell(w)} = v^{\ell(w)}p_{\underline{x}}^w(v)$.
\item The homomorphisms $\{\varphi_{\underline{w}}\circ \LL_{\underline{x},\boldsymbol{e}}\mid \underline{x}^{\boldsymbol{e}} = w\}$ is a basis of $\Hom_{\mathcal{C}}^{\bullet}(B_{\underline{x}},R_w)$.
\end{enumerate}\end{thm}
\begin{proof}
By Proposition~\ref{prop:dual basis of light leaves}, we have
\[
\varphi_{\underline{w}}(\LL_{\underline{x},\boldsymbol{e}}(b_{\underline{x},\boldsymbol{f}})) = 
\begin{cases}
1 & (\boldsymbol{f} = \boldsymbol{e}),\\
0 & (\boldsymbol{f} < \boldsymbol{e}).
\end{cases}
\]
Since $\varphi_{\underline{w}}\circ \LL_{\underline{x},\boldsymbol{e}}\colon B_{\underline{x}}\to R_w$ is a homomorphism in $\mathcal{C}$, it induces $\psi_{\boldsymbol{e}}\colon B_{\underline{x}}^w\to R_w$ and we have
\[
\psi_{\boldsymbol{e}}(b_{\underline{x},\boldsymbol{f}}^w) = 
\begin{cases}
1 & (\boldsymbol{f} = \boldsymbol{e}),\\
0 & (\boldsymbol{f} < \boldsymbol{e}).
\end{cases}
\]
Inductively on $\boldsymbol{e}$, we can take $\psi'_{\boldsymbol{e}} \in \psi_{\boldsymbol{e}} + \sum_{\boldsymbol{e}' > \boldsymbol{e}}R\psi_{\boldsymbol{e}'}$ such that $\psi'_{\boldsymbol{e}}(b_{\underline{x},\boldsymbol{f}}^w) = \delta_{\boldsymbol{e},\boldsymbol{f}}$ (Kronecker's delta).
Namely $m\mapsto \sum_{\underline{x}^{\boldsymbol{e}} = w}\psi'_{\boldsymbol{e}}(m)b_{\underline{x},\boldsymbol{e}}^w$ gives a splitting of the embedding $\bigoplus_{\underline{x}^{\boldsymbol{e}} = w}Rb_{\underline{x},\boldsymbol{e}}^w\hookrightarrow B_{\underline{x}}^w$.

Take $N$ such that $B_{\underline{x}}^w = (\bigoplus_{\underline{x}^{\boldsymbol{e}} = w}Rb_{\underline{x},\boldsymbol{e}}^w)\oplus N$.
Then we have $(B_{\underline{x}}^w)_Q = (\bigoplus_{\underline{x}^{\boldsymbol{e}} = w}Qb_{\underline{x},\boldsymbol{e}}^w)\oplus N_Q$.
We have $\dim_Q (B_{\underline{x}}^w)_Q = p_{\underline{x}}^w(1)$ by Lemma~\ref{lem:Q-dimension of BS module}.
By Lemma~\ref{lem:defect formula}, we also have $\dim_Q(\bigoplus_{\underline{x}^{\boldsymbol{e}} = w}Qb_{\underline{x},\boldsymbol{e}}^w) = p_{\underline{x}}^w(1)$.
Hence $N_Q = 0$.
Namely $N$ is a torsion module.
Since $N\subset B_{\underline{x}}^w\subset (B_{\underline{x}})_Q^w$, $N$ is a torsion-free module.
Hence $N = 0$.

The second part follows from (1).

We prove (3).
Since $\{b_{\underline{x},\boldsymbol{e}}^w\}$ is a basis of $B_{\underline{x}}^w$, $\{\psi'_{\boldsymbol{e}}\}$ is a basis of $\Hom_R(B_{\underline{x}}^w,R)$ which is dual to $\{b_{\underline{x},\boldsymbol{e}}\}$.
By $\psi'_{\boldsymbol{e}} \in \psi_{\boldsymbol{e}} + \sum_{\boldsymbol{e}' > \boldsymbol{e}}R\psi_{\boldsymbol{e}'}$, $\{\psi_{\boldsymbol{e}}\}$ is also a basis.
Since $\Hom_R(B_{\underline{x}}^w,R)\simeq \Hom_{\mathcal{C}}(B_{\underline{x}},R_w)$ and $\psi_{\boldsymbol{e}}$ corresponds to $\varphi_{\underline{w}}\circ \LL_{\underline{x},\boldsymbol{e}}$ by the definition of $\psi_{\boldsymbol{e}}$, we get (3).
\end{proof}

With Lemma~\ref{lem:graded rank of the dual}, we have the following corollary.
\begin{cor}\label{cor:graded rank of B_w}
The left $R$-module $B_{\underline{x},w}$ is graded free and its graded rank is given by $v^{-\ell(w)}p_{\underline{x}}^w(v^{-1})$.
\end{cor}
\begin{cor}\label{cor:top of B_bar w}
The homomorphism $\varphi_{\underline{w}}\colon B_{\underline{w}}\to R_w(\ell(w))$ induces $B_{\underline{w}}^w\simeq R_w(\ell(w))$.
\end{cor}
\begin{proof}
Since $\varphi_{\underline{w}}$ is a homomorphism in $\mathcal{C}$, this induces $B_{\underline{w}}^w\to R_w(\ell(w))$.
This is obviously surjective.
By the above theorem, $B_{\underline{w}}^w$ is free of rank one.
Hence $B_{\underline{w}}^w\twoheadrightarrow R_w(\ell(w))$ is an isomorphism.
\end{proof}

\begin{cor}\label{cor:projectibity of Soergel bimodule}
Let $w\in W$ with a reduced expression $\underline{w}$ and $B\in \Sbimod$.
Then $\Hom_{\mathcal{C}}(B,B_{\underline{w}})\to \Hom_R(B^w,B_{\underline{w}}^w)$ is surjective.
\end{cor}
\begin{proof}
Note that $\Hom_R(B^w,B_{\underline{w}}^w)\simeq \Hom_{\mathcal{C}}(B,B_{\underline{w}}^w)$.
We may assume $B = B_{\underline{x}}$.
The corollary is clear from Theorem~\ref{thm:graded rank of B_x^w} (3) and Corollary~\ref{cor:top of B_bar w}.
\end{proof}

\subsection{Elements supported on a closed subset}
We call a subset $I\subset W$ closed if $w_1\in I$, $w_2\in W$, $w_2\le w_1$ implies $w_2\in I$.
\begin{lem}
Let $I$ be a finite closed subset and $w\in I$ a maximal element with respect to the Bruhat order.
Then there exists a enumeration $w_1,w_2,\ldots$ of elements of $W$ such that $\{w_1,\dots,w_i\}$ is closed for any $i$, $w = w_{\#I}$ and $I = \{w_1,\dots,w_{\#I}\}$.
\end{lem}
\begin{proof}
Set $k = \#I$.
Let $w_1,\dots,w_{k - 1}$ be an enumeration of elements of $I\setminus \{w\}$ such that $w_i\le w_j$ implies $i\le j$.
Let $w_{k + 1},\dots$ be a similar enumeration of elements of $W\setminus I$ and put $w_k = w$.
We prove that $\{w_1,\dots,w_i\}$ is always closed.
If $i \le k$ then it is obvious.
Assume that $i > k$ and $w_j\le w_i$.
If $j\le k$, then we have $j\le i$.
If $j > k$, then we have $j\le i$ by the assumption on the enumeration $w_{k + 1},\dots$.
In any case, $w_j\in \{1,\dots,w_i\}$.
\end{proof}

A light leaf $\LL_{\underline{x},\boldsymbol{e}}\colon B_{\underline{x}}\to B_{\underline{w}}$ gives a homomorphism $B_{\underline{w}}\simeq D(B_{\underline{w}})\xrightarrow{D(\LL_{\underline{x},\boldsymbol{e}})} D(B_{\underline{x}})\simeq B_{\underline{x}}$.
We denote this homomorphism by $\LL_{\underline{x},\boldsymbol{e}}^*$.
Let $\pi_{\underline{x}}^w\colon B_{\underline{x}}\to B_{\underline{x}}^w$ be the projection.
If $I\subset W$ and $w\in I$, then $B_{\underline{x},I\setminus\{w\}}$ is the kernel of $B_{\underline{x},I}\to B_{\underline{x}}^w$.
Therefore we also write $\pi_{\underline{x}}^w$ for the projection $B_{\underline{x},I}\to B_{\underline{x},I}/B_{\underline{x},I\setminus\{w\}}$.

\begin{thm}\label{thm:filtration theorem}
Let $I$ be a closed subset and $w$ a maximal element in $I$.
Set $I' = I\setminus\{w\}$.
Let $\underline{w}$ be a reduced expression of $w$ and $\underline{x}\in S^l$.
Then $\{\pi_{\underline{x}}^w(\LL_{\underline{x},\boldsymbol{e}}^*(u_{\underline{w}}))\mid \underline{x}^{\boldsymbol{e}} = w\}$ is a basis of the left $R$-module $B_{\underline{x},I}/B_{\underline{x},I'}$.
\end{thm}

\begin{proof}
Note that since $\{y\in W\mid y\le w\}\subset I$, we have $\supp_W(B_{\underline{w}})\subset I$.
Hence the image of any homomorphism from $B_{\underline{w}}$ to $B_{\underline{x}}$ is contained in $B_{\underline{x},I}$.
Therefore $\pi_{\underline{x}}^w(\LL^*_{\underline{x},\boldsymbol{e}}(u_{\underline{w}}))\in B_{\underline{x},I}/B_{\underline{x},I'}$.

First we prove that $\{\pi_{\underline{x}}^w(\LL_{\underline{x},\boldsymbol{e}}^*(u_{\underline{w}}))\mid \underline{x}^{\boldsymbol{e}} = w\}$ is linearly independent.
It is sufficient to prove that this set is linearly independent over $Q$.

Recall that we have homomorphisms
\[
B_{\underline{x}}\xrightarrow{\LL_{\underline{x},\boldsymbol{e}}} B_{\underline{w}}\to B_{\underline{w}}^w.
\]
The set of these elements where $\boldsymbol{e}$ satisfies $\underline{x}^{\boldsymbol{e}} = w$ is linearly independent by Theorem~\ref{thm:graded rank of B_x^w}.
Therefore the dualized maps 
\[
B_{\underline{w},w}\hookrightarrow B_{\underline{w}}\xrightarrow{\LL^*_{\underline{x},\boldsymbol{e}}}B_{\underline{x}}
\]
are also linearly independent. (Note that $B_{\underline{w},w}$ and $B_{\underline{x}}$ are both graded free as right $R$-modules, hence $D(D(B_{\underline{w},w}))\simeq B_{\underline{w},w}$ and $D(D(B_{\underline{x}}))\simeq B_{\underline{x}}$.)
This map factors through $B_{\underline{x},w}\hookrightarrow B_{\underline{x}}$.
Therefore the induced homomorphisms $B_{\underline{w},w}\to B_{\underline{x},w}$ are linearly independent.
Since $B_{\underline{x},w}$ is torsion free, the maps $ B_{\underline{w},w}\otimes_R Q\xrightarrow{\eta_{\underline{x},\boldsymbol{e},Q}}B_{\underline{x},w}\otimes_R Q$ obtained by tensoring $Q$ are also linearly independent.

The left $R$-module $B_{\underline{w}}^w$ is free of rank one by Theorem~\ref{thm:graded rank of B_x^w}.
Since $D(B_{\underline{w}}^w)\simeq B_{\underline{w},w}$, this is also true for $B_{\underline{w},w}$.
Therefore the dimension of $B_{\underline{w},w}\otimes_R Q$ and $B_{\underline{w}}^w\otimes_R Q$ are both one and hence we have $B_{\underline{w},w}\otimes_R Q\xrightarrow{\sim}B_{\underline{w}}^w\otimes_R Q$.
Therefore for any $0\ne q\in B_{\underline{w},w}\otimes_R Q\simeq B_{\underline{w}}^w\otimes_R Q$, $\{\eta_{\underline{x},\boldsymbol{e},Q}(q)\mid \underline{x}^{\boldsymbol{e}} = w\}\subset B_{\underline{x},w}\otimes_R Q$ is linearly independent.
In particular we can take $q$ as the image of $u_{\underline{w}}\in B_{\underline{w}}$.
Therefore
\[
\{\eta_{\underline{x},\boldsymbol{e},Q}(\pi_{\underline{w}}^w(u_{\underline{w}}))\mid \underline{x}^{\boldsymbol{e}} = w\}\subset  B_{\underline{x},w}\otimes_R Q\subset (B_{\underline{x},I}/B_{\underline{x},I'})\otimes_R Q
\]
is linearly independent.
Noticing the following commutative diagram
\[
\begin{tikzcd}
&[-2em] B_{\underline{w},w}\otimes_R Q\arrow[r,"\eta_{\underline{x},\boldsymbol{e},Q}"]\arrow[d,"\wr"] & B_{\underline{x},w}\otimes_R Q\arrow[d,hookrightarrow]\\
& B_{\underline{w}}^w\otimes_R Q & (B_{\underline{x},I}/B_{\underline{x},I'})\otimes_R Q\\
u_{\underline{w}} & \arrow[l,phantom,"\in"] B_{\underline{w}}\otimes_R Q\arrow[u,"\pi_{\underline{w}}^w\otimes \Id"]\arrow[r,"\LL^*_{\underline{x},\boldsymbol{e}}\otimes \Id"] & B_{\underline{x},I}\otimes_R Q\arrow[u,"\pi_{\underline{x}}^w\otimes \Id"],
\end{tikzcd}
\]
$\{\pi_{\underline{x}}^w(\LL_{\underline{x},\boldsymbol{e}}^*(u_{\underline{w}}))\mid \underline{x}^{\boldsymbol{e}} = w\}$ is linearly independent over $Q$.

Let $w_1,w_2,\dots$ be an enumeration as in the previous lemma.
Fix a reduced expression $\underline{w}_k$ of $w_k$.
Set $I(k) = \{w_1,\dots,w_k\}$.
We have a filtration $\{B_{\underline{x},I(k)}\}_k$ of $B_{\underline{x}}$ and, as we have proved in the above,
\[
\bigoplus_{\underline{x}^{\boldsymbol{e}} = w_k}R\pi_{\underline{x}}^{w_k}(\LL^*_{\underline{x},\boldsymbol{e}}(u_{\underline{w}_k}))\subset B_{\underline{x},I(k)}/B_{\underline{x},I(k - 1)}.
\]
This gives an embedding in each degree $i\in \Z$:
\begin{equation}
\varphi_{i,k}\colon\bigoplus_{\underline{x}^{\boldsymbol{e}} = w_k}(R\pi_{\underline{x}}^{w_k}(\LL^*_{\underline{x},\boldsymbol{e}}(u_{\underline{w}_k})))^i\hookrightarrow (B_{\underline{x},I(k)}/B_{\underline{x},I(k - 1)})^i.
\end{equation}
We have $\deg(\LL^*_{\underline{x},\boldsymbol{e}}) = d(\boldsymbol{e})$ and $\deg(u_{\underline{w}_k}) = -\ell(w_k)$.
Therefore the graded rank of $\bigoplus_{\underline{x}^{\boldsymbol{e}} = w_k}R\pi_{\underline{x}}^w(\LL^*_{\underline{x},\boldsymbol{e}}(u_{\underline{w}_k}))$ is $\sum_{\underline{x}^{\boldsymbol{e}} = w_k}v^{-d(\boldsymbol{e}) + \ell(w_k)} = v^{\ell(w_k)}p_{\underline{x}}^{w_k}(v^{-1})$.
By Lemma~\ref{lem:sum of p, for filtration}, the sum $\sum_k v^{\ell(w_k)}p_{\underline{x}}^{w_k}(v^{-1})$ is $(v + v^{-1})^l$  and this is equal to the graded rank of $B_{\underline{x}}$ by Lemma~\ref{lem:graded rank of B}.
Hence, if $\Coef$ is a field, $\varphi_{i,k}$ is an isomorphism by Lemma~\ref{lem:filtration, graded rank}.
Therefore, for any (netherian integral domain) $\Coef$, $\varphi_{i,k}\otimes(\Coef/\mathfrak{m})$ is an isomorphism for any maximal ideal $\mathfrak{m}\subset \Coef$.
Since $B_{\underline{x},I(k)}/B_{\underline{x},I(k - 1)}$ is a subquotient of $B_{\underline{x}}$ and $B_{\underline{x}}$ is a finitely generated $R$-module, $B_{\underline{x},I(k)}/B_{\underline{x},I(k - 1)}$ is also a finitely generated $R$-module.
Hence the $i$-th graded piece is finitely generated over $\Coef$.
Therefore $\varphi_{i,k}$ is an isomorphism for all $i,k$.
\end{proof}

For $M\in \mathcal{C}$ and $w\in W$, we put $M_{\le w} = M_{\{y\in W\mid y\le w\}}$.
We define $M_{<w}$ in the obvious way.

\begin{cor}\label{cor:on filtration}
Let $B\in \Sbimod$.
\begin{enumerate}
\item Let $I\subset W$ be a closed subset and $w\in I$ a maximal element. Then we have $B_{\le w}/B_{<w}\xrightarrow{\sim}B_{I}/B_{I\setminus \{w\}}$.
\item Let $\underline{w}$ be a reduced expression of $w\in W$. Then the map $\Hom_{\mathcal{C}}(B_{\underline{w}},B)\to \Hom_R(B_{\underline{w}}^w,B_{\le w}/B_{<w})$ is surjective.
\end{enumerate}
\end{cor}
\begin{proof}
We may assume $B = B_{\underline{x}}$ for some $\underline{x}\in S^l$.
By Theorem~\ref{thm:filtration theorem}, any element in $B_{\underline{x},I}/B_{\underline{x},I\setminus\{w\}}$ has a representative of a form $\sum c_{\boldsymbol{e}}\LL_{\underline{x},\boldsymbol{e}}^*(u_{\underline{w}})$ with $c_{\boldsymbol{e}}\in R$.
Since $\supp_W(u_{\underline{w}})\subset \{y\in W\mid y\le w\}$, $\supp_W(\LL_{\underline{x},\boldsymbol{e}}^*(u_{\underline{w}}))\subset \{y\in W\mid y\le w\}$.
Hence 
$\sum c_{\boldsymbol{e}}\LL_{\underline{x},\boldsymbol{e}}^*(u_{\underline{w}})\in B_{\underline{x},\le w}$.
We get (1).

We prove (2).
The $R$-module $B_{\underline{w}}^w$ is free and $\pi_{\underline{w}}^w(u_{\underline{w}})$ is a basis of this module.
Hence $\Hom_R(B_{\underline{w}}^w,B_{\underline{x},\le w}/B_{\underline{x},<w})\simeq B_{\underline{x},\le w}/B_{\underline{x},<w}$ by $\psi\mapsto\psi(\pi_{\underline{w}}^w(u_{\underline{w}}))$
By the functionality of $B\to B^w$ for $B\in \mathcal{C}$, we have $\psi(\pi_{\underline{w}}^w(u_{\underline{w}})) = \pi_{\underline{x}}^w(\psi(u_{\underline{w}}))$.
Therefore it is sufficient to prove that $\Hom_{\mathcal{C}}(B_{\underline{w}},B_{\underline{x}})\to B_{\underline{x},\le w}/B_{\underline{x},<w}$ defined by $\varphi\mapsto \pi_{\underline{x}}^w(\varphi(u_{\underline{w}}))$ is surjective.
This is clear from Theorem~\ref{thm:filtration theorem}.
\end{proof}

The following proposition is a generalization of \cite[Satz 6.6]{MR2329762}.
\begin{prop}\label{prop:B_w and successive quot}
Let $B\in\Sbimod$ and $w\in W$.
Consider the element $\prod_{tw < w}\alpha_t\in R$ where $t$ runs through reflections in $W$.
Then we have $B_w\xrightarrow{\sim}(\prod_{tw < w}\alpha_t)(B_{\le w}/B_{<w})$.
\end{prop}
\begin{proof}
We may assume $B = B_{\underline{x}}$.
It is well-known that $\#\{t\mid tw < w\} = \ell(w)$.
Therefore $\deg(\prod_{tw < w}\alpha_t) = 2\ell(w)$.
Hence by Corollary~\ref{cor:graded rank of B_w} and Theorem~\ref{thm:filtration theorem}, both sides are graded free with the same graded rank.

First we assume that $\underline{x}$ is a reduced expression of $w$.
We denote $\underline{x}$ by $\underline{w}$.
In this case, $B_{\underline{w},\le w}/B_{\underline{w},<w}\simeq B_{\underline{w}}^w$.
With the argument in the last part of the proof of Theorem~\ref{thm:filtration theorem}, it is sufficient to prove $B_{\underline{w},w}\subset (\prod_{tw < w}\alpha_t)B_{\underline{w}}^w$.
We prove this by induction on $\ell(w)$.

Let $\underline{w} = (s_1,\dots,s_l)$ and set $s = s_1$, $\underline{sw} = (s_2,\dots,s_l)$.
Then $\underline{sw}$ is a reduced expression of $sw$.
Let $\delta_s\in V$ such that $\langle\alpha_s^\vee,\delta_s\rangle = 1$.
Take $1\otimes m + \delta_s\otimes m'\in R\otimes_{R^s}B_{\underline{sw}} = B_{\underline{w}}$ and assume that $1\otimes m + \delta_s\otimes m'\in B_{\underline{w},w}$.
By Lemma~\ref{lem:support of B_s otimes}, $\supp_W(m')\in \{w,sw\}$.
We also have $m'_w = 0$ since $B_{\underline{sw}}^w = 0$.
Therefore $m'\in B_{\underline{sw},sw}$.
By inductive hypothesis, we have $m'\in (\prod_{tsw < sw}\alpha_t)B_{\underline{sw}}^{sw}$.

Since $1\otimes m + \delta_s\otimes m'\in B_{\underline{w},w}$, $(1\otimes m + \delta_s\otimes m')_{sw} = 0$.
Hence we have $m_{sw} + \delta_s m'_{sw} = 0$.
Therefore $m_{sw} + s(\delta_s) m'_{sw} = (-\delta_s + s(\delta_s))m'_{sw} = -\alpha_s m'_{sw}\in \alpha_s (\prod_{tsw < sw}\alpha_t)B_{\underline{sw}}^{sw}$.
It is well-known that $\{t\mid tw < w\} = \{s\}\cup \{sts^{-1}\mid tsw < sw\}$.
Hence $\prod_{tw < w}\alpha_t = \alpha_s s(\prod_{tsw < sw}\alpha_t) = s(-\alpha_s\prod_{tsw < sw}\alpha_t)$.
Therefore $m_{sw} + s(\delta_s) m'_{sw}\in s(\prod_{tw < w}\alpha_t)B_{\underline{sw}}^{sw}$.
Take $n\in B_{\underline{sw}}$ such that $m_{sw} + s(\delta_s) m'_{sw} = s(\prod_{tw < w}\alpha_t)n_{sw}$.
Now we have $(1\otimes m + \delta_s\otimes m')_w = (m_w + \delta_s m'_w,m_{sw} + s(\delta_s)m'_{sw})$ and since $m_w,m'_w\in B_{\underline{sw}}^w = 0$, we get $(1\otimes m + \delta_s\otimes m')_w = (0,s(\prod_{tw < w}\alpha_t)n_{sw}) = (\prod_{tw < w}\alpha_t)(0,n_{sw})$.
(Recall that the left action of $R$ on the second factor of $(B_s\otimes B_{\underline{sw}})_Q^w = (B_{\underline{sw}})_Q^w\oplus(B_{\underline{sw}})_Q^{sw}$ is twisted by $s$.)
Again using $n_{w}\in B_{\underline{sw}}^w = 0$, we have $(\prod_{tw < w}\alpha_t)(0,n_{sw}) = (\prod_{tw < w}\alpha_t)(n_w,n_{sw}) = (\prod_{tw < w}\alpha_t)(1\otimes n)_w\in (\prod_{tw < w}\alpha_t)B_ {\underline{w}}^w$.

We prove the proposition for $B = B_{\underline{x}}$ with $\underline{x}\in S^l$.
It is sufficient to prove that $(\prod_{tw < w}\alpha_t)B_{\underline{x},\le w}/B_{\underline{x},<w}\subset B_{\underline{x},w}$.
Set $f = \prod_{tw < w}\alpha_t$ and let $fm$ be an element of the left hand side.
Then we may assume $m = \LL^*_{\underline{x},\boldsymbol{e}}(\pi_{\underline{w}}^w(u_{\underline{w}}))$ by Theorem~\ref{thm:filtration theorem}.
Since we have proved the proposition for $B = B_{\underline{w}}$, $f\pi_{\underline{w}}^w(u_{\underline{w}})\in B_{\underline{w},w}$.
Hence $fm = \LL^*_{\underline{x},\boldsymbol{e}}(f\pi_{\underline{w}}^w(u_{\underline{w}}))\in B_{\underline{x},w}$.
\end{proof}

\section{The categorification theorem}\label{sec:The categorification theorem}
In this section, we assume that $\Coef$ is a complete local ring.
Therefore a direct summand of a graded free $R$-module is again graded free.
\subsection{The classification of indecomposable objects}
\begin{thm}\label{thm:classification}
\begin{enumerate}
\item For each $w\in W$, there exists an indecomposable object $B(w)\in\Sbimod$ such that $\supp_W(B(w))\subset \{x\in W\mid x\le w\}$ and $B(w)^w\simeq R_w(\ell(w))$.
Moreover $B(w)$ is unique up to isomorphism.
\item For any indecomposable object $B\in \Sbimod$ there exists unique $(w,n)\in W\times \Z$ such that $B\simeq B(w)(n)$.
\item We have $D(B(w))\simeq B(w)$.
\item For a reduced expression $\underline{w}$ of $w\in W$, we have $B_{\underline{w}} = B(w)\oplus \bigoplus_{y < w,n\in\Z}B(y)(n)^{m_{n,y}}$ for some $m_{n,y}\in\Z_{\ge 0}$.
\end{enumerate}
\end{thm}
\begin{proof}
Fix a reduced expression $\underline{w}$ of $w$.
Then we have $B_{\underline{w}}^w\simeq R_w(\ell(w))$.
Therefore there exists a unique indecomposable direct summand $B(w)$ of $B_{\underline{w}}$ such that $B(w)^w = B_{\underline{w}}^w(\ell(w))$.
This satisfies the condition of (1) and since $D(B_{\underline{w}})\simeq B_{\underline{w}}$, we have $D(B(w))\simeq B(w)$.

It only remains to show that any object in $\Sbimod$ is a direct sum of $B(w)(n)$.
Let $B\in \Sbimod$.
Take $w\in W$ such that $B^w\ne 0$ and $\ell(w)$ is maximal with respect to this condition.
Set $I = \{y\in W\mid \ell(y)\le \ell(w)\}$.
This is a closed subset of $W$.
The condition of $w$ and the definition of $I$ says that $B_{I} = B$.
Hence $B_I/B_{I\setminus\{w\}} = B^w$.

Since $B^w$ is graded free and $B_{\underline{w}}^w\simeq R_w(\ell(w))$, there exists $n\in \Z$ such that $B_{\underline{w}}(n)^w$ is a direct summand of $B^w$.
Let $i\colon B_{\underline{w}}(n)^w\hookrightarrow B^w$ and $p\colon B^w\twoheadrightarrow B_{\underline{w}}(n)^w$ be the embedding from and the projection to the direct summand.
Then by Corollary~\ref{cor:projectibity of Soergel bimodule} and \ref{cor:on filtration} (2), there exists $B_{\underline{w}}(n)\to B$  and $B\to B_{\underline{w}}(n)$ which is a lift of $i$ and $p$, respectively.
Since $B(w)(n)$ is a direct summand such that $B(w)(n)^w = B_{\underline{w}}(n)^w$, composing the embedding from or the projection to $B(w)(n)$, we get a lift $\widetilde{i}\colon B(w)(n)\to B$ and $\widetilde{p}\colon B\to B(w)(n)$.
The composition $\widetilde{p}\circ\widetilde{i}$ is identity on $B(w)(n)^w$, hence $\widetilde{p}\circ\widetilde{i}\in \End(B(w)(n))$ is not nilpotent.
Since $B(w)(n)$ is indecomposable, $\End(B(w)(n))$ is local.
Hence $\widetilde{p}\circ\widetilde{i}$ is an isomorphism.
Therefore $B(w)(n)$ is a direct summand of $B$.
The last statement follows from the argument with supports.
\end{proof}

It is easy to see that $B(s) = B_s$ for any $s\in S$.

\subsection{The categorification}
Let $[\Sbimod]$ be the split Grothendieck group of the category $\Sbimod$.
Then this has a structure of $\Z[v^{\pm 1}]$-algebra via $v[B] = [B(1)]$ and $[B_1][B_2] = [B_1\otimes B_2]$ where $[B]$  is the image of $B\in \Sbimod$ in $[\Sbimod]$.
By Theorem~\ref{thm:classification}, $\{[B(w)]\mid w\in W\}$ is a $\Z[v^{\pm 1}]$-basis of $[\Sbimod]$.
By Theorem~\ref{thm:classification} (4), $\{[B_{\underline{w}}]\mid w\in W\}$ is also a $\Z[v^{\pm 1}]$-basis of $[\Sbimod]$, here we fix a reduced expression $\underline{w}$ for each $w\in W$.
In particular, $\{[B_{\underline{x}}]\mid \underline{x}\in S^l,l\in\Z_{\ge 0}\}$ generates $[\Sbimod]$

By Theorem~\ref{thm:graded rank of B_x^w}, the $R$-module $B_{\underline{x}}^w$ is graded free.
Therefore is the $R$-module $B^w$ for any $B\in \Sbimod$ is graded free.
We define the character map $\ch\colon [\Sbimod]\to \mathcal{H}$ by
\[
\ch(B) = \sum_{w\in W}v^{-\ell(w)}\grk(B^w)H_w.
\]

\begin{prop}\label{prop:image of character map of BS module}
$\ch(B_{\underline{x}}) = \underline{H}_{\underline{x}}$.
\end{prop}
\begin{proof}
Clear from Theorem~\ref{thm:graded rank of B_x^w}.
\end{proof}

In particular, $\ch$ is an algebra homomorphism on $\sum_{\underline{x}\in S^l,l\in\Z_{\ge 0}}\Z[v^{\pm 1}][B_{\underline{x}}]$.
Since $\{[B_{\underline{x}}]\mid \underline{x}\in S^l,l\in\Z_{\ge 0}\}$ generates $[\Sbimod]$, $\ch$ is an algebra homomorphism from $[\Sbimod]$ to $\mathcal{H}$.

For a reduced expression $\underline{w}$ of each $w\in W$, we have $\underline{H}_{\underline{w}}\in H_w + \sum_{y < w}\Z[v^{\pm 1}]H_y$.
Hence $\{\underline{H}_{\underline{w}}\mid w\in W\}$ is a basis of $\mathcal{H}$.
Since $\ch$ sends a basis $\{[B_{\underline{w}}]\mid w\in W\}$ to a basis $\{\underline{H}_{\underline{w}}\mid w\in W\}$, we get the following categorification theorem.
\begin{thm}
The algebra homomorphism $\ch$ is an isomorphism $[\Sbimod]\simeq \mathcal{H}$.
\end{thm}

\subsection{A formula on homomorphisms}
We define:
\begin{itemize}
\item an involution $h\mapsto \overline{h}$ on $\mathcal{H}$ by $\overline{\sum_{w\in W}a_w(v)H_w} = \sum_{w\in W}a_w(v^{-1})H_{w^{-1}}^{-1}$.
\item an anti-involution $\omega\colon \mathcal{H}\to \mathcal{H}$ by $\omega(\sum_{w\in W}a_w(v)H_w) = \sum_{w\in W}a_w(v^{-1})H_w^{-1}$.
\item a $\Z[v^{\pm 1}]$-linear map $\varepsilon\colon \mathcal{H}\to \Z[v^{\pm 1}]$ by $\varepsilon(\sum_{w\in W}a_wH_w) = a_e$.
We also put $\overline{\varepsilon}(h) = \overline{\varepsilon(\overline{h})}$.
\end{itemize}
The following lemma follows from a straightforward calculation.
\begin{lem}
The linear map $\varepsilon$ is a trace, namely it satisfies $\varepsilon(hh') = \varepsilon(h'h)$ for any $h,h'\in \mathcal{H}$.
The linear map $\overline{\varepsilon}$ is also a trace.
\end{lem}

\begin{lem}\label{lem:a lemma on the polynomial p}
Let $s_1,\dots,s_l\in S$.
\begin{enumerate}
\item We have $\omega(\underline{H}_{(s_1,\dots,s_l)}) = \underline{H}_{(s_l,\dots,s_1)}$ and $\overline{\underline{H}_{(s_1,\dots,s_l)}} = \underline{H}_{(s_1,\dots,s_l)}$.
\item We have $p_{(s_l,\dots,s_1)}^{w^{-1}} = p_{(s_1,\dots,s_l)}^w$ for any $w\in W$.
\end{enumerate}
\end{lem}
\begin{proof}
A direct calculation shows that $\overline{\underline{H}_s} = \underline{H}_s$ and $\omega(\underline{H}_s) = \underline{H}_s$.
The first part follows from this.
For the second, we have
\begin{align*}
\sum_{w\in W}p_{(s_l,\dots,s_1)}^{w^{-1}}(v)H_{w} & = \sum_{w\in W}p_{(s_l,\dots,s_1)}^{w}(v)H_{w^{-1}}\\
& = \overline{\omega\left(\sum_{w\in W}p_{(s_l,\dots,s_1)}^{w}(v)H_w\right)}\\
& = \overline{\omega(H_{(s_l,\dots,s_1)})}\\
& = H_{(s_1,\dots,s_l)}\\
& = \sum_{w\in W}p_{(s_1,\dots,s_l)}^wH_w.
\end{align*}
Comparing the coefficient of $H_w$, we get the lemma.
\end{proof}

\begin{thm}
Let $B_1,B_2\in\Sbimod$.
Then $\Hom_{\Sbimod}(B_1,B_2)$ is graded free as a left $R$-module and its graded rank is given by
\[
\grk(\Hom_{\Sbimod}(B_1,B_2)) = \overline{\varepsilon}(\omega(\ch(B_1))\ch(B_2))).
\]
\end{thm}
\begin{proof}
For $B_1 = B_{(s_1,\dots,s_l)}$ and $B_2 = B_{(t_1,\dots,t_r)}$, by Lemma~\ref{lem:B_s is self-adjoint}, we have 
\begin{equation}\label{eq:adjoint and hom-formula}
\Hom_{\Sbimod}(B_1,B_2) \simeq \Hom_{\Sbimod}(R_e,B_{(s_l,\dots,s_1)}B_{(t_1,\dots,t_r)}) \simeq (B_{(s_l,\dots,s_1,t_1,\dots,t_r)})_e
\end{equation}
and this is a graded free left $R$-module by Corollary~\ref{cor:graded rank of B_w}.
Therefore $\Hom_{\Sbimod}(B_1,B_2)$ is graded free for any $B_1,B_2\in \Sbimod$.

The map $(B_1,B_2)\mapsto \grk(\Hom_{\Sbimod}(B_1,B_2)) -  \overline{\varepsilon}(\omega(\ch(B_1))\ch(B_2))$ defines a bilinear form on the $\Z$-module $[\Sbimod]\simeq \mathcal{H}$ which we denote by $f$.
The following properties follow from a straightforward calculation.
\begin{gather}
f(v^{-1}h_1,h_2) = f(h_1,vh_2) = vf(h_1,h_2),\label{eq:paring, on v}\\
f(\underline{H}_sh_1,h_2) = f(h_1,\underline{H}_sh_2).\label{eq:paring, self-adjointness}
\end{gather}

If $B_1 = R_e$ and $B_2 = B_{\underline{x}}$, then $\Hom_{\Sbimod}(B_1,B_2) = (B_{\underline{x}})_e$ has the graded rank $p_{\underline{x}}^e(v^{-1})$ by Corollary~\ref{cor:graded rank of B_w}.
By Proposition~\ref{prop:image of character map of BS module} and $\overline{\underline{H}_{\underline{x}}} = \underline{H}_{\underline{x}}$, $\overline{\varepsilon}(\omega(\ch(R_e))\ch(B_{\underline{x}}))) = \overline{\varepsilon(\overline{\underline{H}_{\underline{x}}})} = \overline{\varepsilon(\underline{H}_{\underline{x}})} = \overline{p_{\underline{x}}^e(v)} = p_{\underline{x}}^e(v^{-1})$.
Hence $f(1,\underline{H}_{\underline{x}}) = 0$.
Therefore by \eqref{eq:paring, self-adjointness}, we have $f(\underline{H}_{(s_1,\dots,s_l)},\underline{H}_{(t_1,\dots,t_r)}) = f(1,\underline{H}_{(s_l,\dots,s_1,t_1,\dots,t_r)}) = 0$.
Since $\{\underline{H}_{\underline{x}}\}$ spans $\mathcal{H}$ as a $\Z[v^{\pm 1}]$-module, by \eqref{eq:paring, on v}, $f = 0$ on $\mathcal{H}\times \mathcal{H}$.
\end{proof}

\begin{cor}\label{cor:hom between BS modules}
The graded rank of $\Hom_{\Sbimod}(B_{\underline{x}},B_{\underline{y}})$ is $\sum_{w\in W}p_{\underline{x}}^w(v^{-1})p_{\underline{y}}^w(v^{-1})$.
\end{cor}
\begin{proof}
Let $\underline{x} = (s_1,\dots,s_l)$ and set $\underline{x'} = (s_l,\dots,s_1)$.
Then we have $\omega(\underline{H}_{\underline{x}}) = \underline{H}_{\underline{x'}}$.
Hence $\overline{\varepsilon}(\omega(\underline{H}_{\underline{x}})\underline{H}_{\underline{y}}) = \overline{\varepsilon}(\underline{H}_{\underline{x'}}\underline{H}_{\underline{y}})$.
Since $\overline{\underline{H}_{\underline{x'}}} = \underline{H}_{\underline{x'}}$ and $\overline{\underline{H}_{\underline{y}}} = \underline{H}_{\underline{y}}$, we have $\overline{\varepsilon}(\underline{H}_{\underline{x'}}\underline{H}_{\underline{y}}) = \overline{\varepsilon(\underline{H}_{\underline{x'}}\underline{H}_{\underline{y}})}$.
We have
\begin{align*}
\varepsilon(\underline{H}_{\underline{x'}}\underline{H}_{\underline{y}}) & = \varepsilon\left(\left(\sum_{w_1\in W}p_{\underline{x'}}^{w_1}H_{w_1}\right)\left(\sum_{w_2\in W}p_{\underline{y}}^{w_2}H_{w_2}\right)\right)\\
& = \sum_{w_1,w_2\in W}p_{\underline{x'}}^{w_1}p_{\underline{y}}^{w_2}\varepsilon(H_{w_1}H_{w_2}).
\end{align*}
We have $\varepsilon(H_{w_1}H_{w_2}) = \delta_{w_1^{-1},w_2}$ \cite[(4.3)]{MR2441994}.
Since $p_{\underline{x'}}^{w_1} = p_{\underline{x}}^{w_1^{-1}}$ by Lemma~\ref{lem:a lemma on the polynomial p}, we get the corollary.
\end{proof}

\section{Relations with other categories}\label{sec:Relations with other categories}
\subsection{Sheaves on moment graphs}\label{subsec:Sheaves on moment graphs}
Define an algebra $\mathcal{Z}$ by
\[
\mathcal{Z} = \left\{(z_w)\in \prod_{w\in W}R\mid \text{$z_{tw}\equiv z_w\pmod{\alpha_t}$ for any $w\in W$ and reflection $t$ }\right\}.
\]
This is an $R$-algebra via $f(z_w) = (fz_w)$ for $f\in R$ and $(z_w)\in \mathcal{Z}$.
This is the structure algebra of the moment graph attached to $(W,S)$.
Fiebig developed the theory of sheaves on moment graphs, see \cite{MR2395170,MR2370278}.
In particular, he proved that the category of Soergel bimodules in the original sense is equivalent to a certain full-subcategory of $\mathcal{Z}$-modules when the representation $V$ is reflection faithful.
We generalize it.

If $f\in R$, then $(w(f))_{w\in W}\in \mathcal{Z}$.
For any graded $\mathcal{Z}$-module $M$, we define a right action of $f$ as the action of $(w(f))_{w\in W}\in \mathcal{Z}$.
Hence $M$ is an $R$-bimodule.
To make a graded $\mathcal{Z}$-module $M$ into an object of $\mathcal{C}'$, we need a finiteness assumption on $M$.

For a subset $I\subset W$, set
\[
\mathcal{Z}^I = \left\{(z_w)\in \prod_{w\in I}R\;\middle|\; 
\begin{array}{l}
\text{$z_{tw}\equiv z_w\pmod{\alpha_t}$ for any $w\in I$}\\
\text{and any reflection $t\in W$ such that $tw\in I$}
\end{array}
\right\}.
\]
We have a canonical homomorphism $\mathcal{Z}\to \mathcal{Z}^I$.
Let $\ModCat{\mathcal{Z}}^f$ be the full-subcategory of the category of graded $\mathcal{Z}$-modules which consists of modules such that the action of $\mathcal{Z}$ factors through $\mathcal{Z}\to \mathcal{Z}^I$ for some finite $I\subset W$.

For a $\mathcal{Z}$-module $M$, let $M_Q = Q\otimes_R M$ and set
\[
M_Q^x = \{m\in M_Q\mid \text{$(z_w)m = z_xm$ for any $(z_w)\in \mathcal{Z}$}\}.
\]
Then if $M\in \ModCat{\mathcal{Z}}^f$, we have $M_Q = \bigoplus_{x\in W}M_Q^x$ \cite[3.2]{MR2395170}.
It is easy to see that this defines an object of $\mathcal{C}'$.
Hence this gives a functor $F\colon\ModCat{\mathcal{Z}}^f\to \mathcal{C}'$.
\begin{prop}
Let $M,N\in \ModCat{\mathcal{Z}}^f$ and assume that $N$ is torsion-free as an $R$-module.
Then we have $\Hom_{\ModCat{\mathcal{Z}}^f}(M,N)\xrightarrow{\sim}\Hom_{\mathcal{C}'}(F(M),F(N))$.
\end{prop}
\begin{proof}
Let $\varphi\colon F(M)\to F(N)$ be a homomorphism in $\mathcal{C}'$.
Then $\varphi$ induces a map $\varphi_w\colon M_Q^w\to N_Q^w$.
For $m\in M$, $z = (z_w)\in \mathcal{Z}$ and $w\in W$, we have $\varphi(zm)_w = \varphi_w((zm)_w) = \varphi_w(z_wm_w) = z_w\varphi_w(m_w) = z_w(\varphi(m)_w) = (z\varphi(m))_w$.
Since the map $N\to N_Q = \bigoplus_{w\in W}N_Q^w$ is injective by the assumption, we have $\varphi(zm) = z\varphi(m)$.
\end{proof}

Let $s\in S$ and define $\mathcal{Z}^s = \{(z_w)\in \mathcal{Z}\mid z_{ws} = z_w\}$.
This is a subalgebra of $\mathcal{Z}$.
We say that $V$ satisfies the GKM condition if $\{\alpha_{t_1},\alpha_{t_2}\}$ is linearly independent for any reflections $t_1\ne t_2$.
\begin{lem}
Assume that $V$ satisfies GKM condition.
Let $s\in S$ and $\delta\in V$ such that $\langle \alpha_s^\vee,\delta\rangle = 1$.
Then we have $\mathcal{Z} = \mathcal{Z}^s\oplus (w(\delta))_{w\in W}\mathcal{Z}^s$.
\end{lem}
\begin{proof}
Let $z = (z_w)_{w\in W}$.
For each $w\in W$, there exists $y_w\in R$ such that $z_{w} - z_{ws} = w(\alpha_s)y_w$.
Let $t\ne wsw^{-1}$ be a reflection.
Then $tw(\alpha_s)\equiv w(\alpha_s)\pmod{\alpha_t}$.
Hence $w(\alpha_s)(y_w - y_{tw})\equiv w(\alpha_s)y_w - tw(\alpha_s)y_{tw} = (z_w - z_{tw}) - (z_{ws} - z_{tws})\equiv 0\pmod{\alpha_t}$.
By the GKM condition, $\alpha_t$ and $w(\alpha_s)$ are linearly independent.
Hence $y_w\equiv y_{tw}\pmod{\alpha_t}$.
On the other hand, for $t = wsw^{-1}$, we have $y_{tw} = y_{ws} = y_w$ by the definition of $y_w$.
Hence $y = (y_w)_{w\in W}\in \mathcal{Z}^s$.
We also set $x_w = z_w - w(\delta)y_w$.
Then we have $x_{ws} = z_{ws} - ws(\delta)y_w = z_w - w(\alpha_s)y_w - ws(\delta)y_w = z_w - w(\alpha_s)y_w - (w(\delta) - w(\alpha_s))y_w = z_w - w(\delta)y_w = x_w$.
Hence $x = (x_w)_{w\in W}\in \mathcal{Z}^s$ and we have $z = x + (w(\delta))_{w\in W}y$.
Reversing this argument, we can get the uniqueness of $x,y$.
\end{proof}
\begin{prop}
Assume that $V$ satisfies GKM condition.
Let $M\in \ModCat{Z}^f$.
Then $F(\mathcal{Z}\otimes_{\mathcal{Z}^s}M)\simeq F(M)\otimes_R B_s$.
\end{prop}
\begin{proof}
Take $\delta\in V$ such that $\langle\alpha_s^\vee,\delta\rangle = 1$.
Define $F(M)\otimes_R B_s = F(M)\otimes_{R^s}R\to \mathcal{Z}\otimes_{\mathcal{Z}^s}M$ by $m\otimes f\mapsto (w(f))_{w\in W}\otimes m$.
Since $F(M)\otimes_{R^s}R = F(M)\otimes 1\oplus F(M)\otimes\delta$ and $\mathcal{Z}\otimes_{\mathcal{Z}^s}M = 1\otimes M\oplus (w(\delta))_{w\in W}\otimes M$, this homomorphism is an isomorphism.
\end{proof}

Define a $\mathcal{Z}$-module structure on $R$ by $(z_w)_{w\in W}f = z_ef$ for $(z_w)_{w\in W}\in \mathcal{Z}$ and $f\in R$ and denote this $\mathcal{Z}$-module by $R(e)$.
Then $F(R(e)) = R_e$.
Let $\ModCat{\mathcal{Z}}^{\textrm{S}}$ be the full-subcategory of $\ModCat{\mathcal{Z}}^f$ consisting of the direct summands of direct sums of $\{\mathcal{Z}\otimes_{\mathcal{Z}^{s_1}}\dotsm\otimes_{\mathcal{Z}^{s_l}}R(e)(n)\mid s_1,\dots,s_l\in S,n\in\Z\}$.
The following theorem follows from the above argument.
\begin{thm}
Assume that $V$ satisfies GKM condition.
The functor $F$ induces an equivalence $\ModCat{\mathcal{Z}}^{\mathrm{S}}\to \Sbimod$.
\end{thm}

\subsection{Double leaves}
Let $\underline{x}\in S^l$, $\underline{y}\in S^{l'}$, $\boldsymbol{e}\in \{0,1\}^l$ and $\boldsymbol{f}\in \{0,1\}^{l'}$ such that $\underline{x}^{\boldsymbol{e}} = \underline{y}^{\boldsymbol{f}}$.
Set $w = \underline{x}^{\boldsymbol{e}}$ and fix its reduced expression $\underline{w}$.
Then we have $\LL_{\underline{x},\boldsymbol{e}}\colon B_{\underline{x}}\to B_{\underline{w}}$ and $\LL^*_{\underline{y},\boldsymbol{f}}\colon B_{\underline{w}}\to B_{\underline{y}}$.
We put $\DLL_{\boldsymbol{e},\boldsymbol{f}} = \LL^*_{\underline{y},\boldsymbol{f}}\circ \LL_{\underline{x},\boldsymbol{e}}\colon B_{\underline{x}}\to B_{\underline{y}}$ which we call a double leaf.

\begin{thm}\label{thm:double leaves}
The set of double leaves $\{\DLL_{\boldsymbol{e},\boldsymbol{f}}\mid \underline{x}^{\boldsymbol{e}} = \underline{y}^{\boldsymbol{f}}\}$ is a basis of $\Hom_{\BSbimod}(B_{\underline{x}},B_{\underline{y}})$.
\end{thm}
\begin{proof}
The degree of $\DLL_{\boldsymbol{e},\boldsymbol{f}}$ is $d(\boldsymbol{e}) + d(\boldsymbol{f})$.
Therefore we have $\sum_{\underline{x}^{\boldsymbol{e}} = \underline{y}^{\boldsymbol{f}}}v^{-\deg(\DLL_{\boldsymbol{e},\boldsymbol{f}})} = \sum_{w\in W}\sum_{\underline{x}^{\boldsymbol{e}} = w}v^{-d(\boldsymbol{e})}\sum_{\underline{x}^{\boldsymbol{f}} = w}v^{-d(\boldsymbol{f})} = \sum_{w\in W}p_{\underline{x}}^w(v^{-1})p_{\underline{y}}^w(v^{-1})$.
This is equal to the graded rank of $\Hom_{\BSbimod}(B_{\underline{x}},B_{\underline{y}})$ by Corollary~\ref{cor:hom between BS modules}.
Therefore it is sufficient to prove that the set of double leaves is linearly independent.

Assume that $\sum c_{\boldsymbol{e},\boldsymbol{f}}\DLL_{\boldsymbol{e},\boldsymbol{f}} = 0$.
Set $I = \{\underline{x}^{\boldsymbol{e}}\mid \underline{x}^{\boldsymbol{e}} = \underline{y}^{\boldsymbol{f}},c_{\boldsymbol{e},\boldsymbol{f}}\ne 0\}$ and assume that $I$ is not empty.
Let $w\in I$ be an element with maximal length.
Let $\overline{I} = \{x\in W\mid \text{there exists $y\in I$ such that $x\le y$}\}$.
Then $w$ is also has maximal length in $\overline{I}$ and $\overline{I}$ is closed.
Set $\overline{I}' = \overline{I}\setminus\{w\}$.
Then $c_{\boldsymbol{e},\boldsymbol{f}}\DLL_{\boldsymbol{e},\boldsymbol{f}}$ factors through $B_{\underline{y},\overline{I}}\hookrightarrow B_{\underline{y}}$.
Since $\overline{I}'\supset I\setminus\{w\}$, we have $\pi_{\underline{y}}^w\circ (c_{\boldsymbol{e},\boldsymbol{f}}\DLL_{\boldsymbol{e},\boldsymbol{f}}) = 0$ unless $\underline{x}^{\boldsymbol{e}} = w$.

Set $E = \{\boldsymbol{e}\mid c_{\boldsymbol{e},\boldsymbol{f}}\ne 0, \underline{x}^{\boldsymbol{e}} = w\}$.
Recall that we have a total order ${<} = {<_{\underline{x},w}}$ on this set.
Let $\boldsymbol{e}'\in E$ be the minimal element.
Then for $\boldsymbol{e}\in E$, we have $\LL_{\underline{x},\boldsymbol{e}}(b_{\underline{x},\boldsymbol{e}'}) = 0$ unless $\boldsymbol{e} = \boldsymbol{e}'$ by Proposition~\ref{prop:dual basis of light leaves}.
Therefore we have $\sum_{\underline{y}^{\boldsymbol{f}} = w}c_{\boldsymbol{e}',\boldsymbol{f}}(\pi_{\underline{y}}^w(\DLL_{\boldsymbol{e}',\boldsymbol{f}}(b_{\underline{x},\boldsymbol{e}'}))) = 0$.
Since $\LL_{\underline{x},\boldsymbol{e}'}(b_{\underline{x},\boldsymbol{e}'}) = u_{\underline{w}}$ (Proposition~\ref{prop:dual basis of light leaves}), we have $\sum_{\underline{y}^{\boldsymbol{f}} = w}{c_{\boldsymbol{e}',\boldsymbol{f}}}\pi_{\underline{y}}^w(\LL^*_{\underline{y},\boldsymbol{f}}(u_{\underline{w}})) = 0$.
By Theorem~\ref{thm:filtration theorem}, $c_{\boldsymbol{e}',\boldsymbol{f}} = 0$ for any $\boldsymbol{f}$ such that $\underline{y}^{\boldsymbol{f}} = w$.
This is a contradiction.
\end{proof}

\subsection{The category of Elias-Williamson}
Let $\mathcal{D}$ be the category defined in \cite{MR3555156}.
We also use notation as in \cite{MR3555156}.
In this subsection, we assume the following.
\begin{itemize}
\item Assumptions on $V$ in \cite{MR3555156}.
\item The homomorphism attached to a $2m$-valent graph in \cite[Definition~5.13]{MR3555156} is a homomorphism in $\mathcal{C}$ and sends $u_{\bullet}$ to $u_{\bullet}$, namely gives a homomorphism as in Assumption~\ref{assump:existecne of 2-coloerd map}.
\end{itemize}
If the restriction of $V$ to the group generated by $\{s,t\}$ is fully-faithful for any $s,t\in S$ such that $st$ has the finite order, then these conditions hold.

\begin{thm}
The category $\mathcal{D}$ is equivalent to $\BSbimod$.
Therefore the category $\Kar(\mathcal{D})$ is equivalent to $\Sbimod$.
\end{thm}
\begin{proof}
In \cite[Definition 5.13]{MR3555156}, the functor from $\mathcal{D}$ to the category of $R$-bimodules is constructed.
This functor sends $\underline{x}\in \mathcal{D}$ to $B_{\underline{x}}$ which can be regarded as a map from the objects of $\mathcal{D}$ to those of $\BSbimod$.
Each generator of homomorphisms in $\mathcal{D}$ is sent to a homomorphism in $\mathcal{C}$.
The relations in $\mathcal{D}$ are satisfied in the category of $R$-bimodules~\cite[Claim 5.14]{MR3555156}.
Since $\BSbimod$ is a full subcategory of the category of $R$-bimodules, the relations are preserved in $\BSbimod$.
Hence we have a functor from $\mathcal{D}$ to $\BSbimod$.
This functor is obviously essentially surjective.
It sends double leaves to double leaves and the sets of double leaves give a basis of the space of homomorphisms in both categories (Theorem~\ref{thm:double leaves} and \cite[Theorem~6.12]{MR3555156}).
Hence the functor is fully-faithful.
\end{proof}

\end{document}